\documentclass[11pt]{amsart}
\usepackage[utf8]{inputenc} 
\usepackage[british]{babel} 
\usepackage{culmus} 
\usepackage{dsfont}

\begin{hyphenrules}{british}
\end{hyphenrules}

\usepackage{graphicx} 
\usepackage{booktabs} 
\usepackage{array} 
\usepackage{paralist} 
\usepackage{verbatim} 
\usepackage{subfig} 
\usepackage{scalerel,stackengine} 

\usepackage{bmpsize}
\usepackage{amssymb}
\usepackage{mathtools}
\usepackage{amsmath}	
\usepackage{amssymb}
\usepackage{amsthm}
\usepackage[mathcal]{eucal}
\usepackage{graphicx}
\usepackage{faktor} 
\usepackage{mathrsfs} 
\usepackage[all, cmtip]{xy} 
\usepackage{systeme}
\usepackage[titletoc]{appendix}
\usepackage{leftindex}
\usepackage{enumitem}
\usepackage{appendix}

\swapnumbers

\theoremstyle{plain}
\newtheorem{theorem}[subsubsection]{Theorem}
\newtheorem{lemma}[subsubsection]{Lemma}
\newtheorem{proposition}[subsubsection]{Proposition}
\newtheorem{corollary}[subsubsection]{Corollary}

\theoremstyle{definition}
\newtheorem{definition}[subsubsection]{Definition}

\theoremstyle{remark}

\newtheorem{examples}[subsubsection]{Examples}
\newtheorem{remark}[subsubsection]{Remark}
\newtheorem{remarks}[subsubsection]{Remarks}
\newtheorem{remarks-examples}[subsubsection]{Remarks and Examples}

\stackMath
\newcommand\reallywidehat[1]{%
\savestack{\tmpbox}{\stretchto{%
  \scaleto{%
    \scalerel*[\widthof{\ensuremath{#1}}]{\kern-.6pt\bigwedge\kern-.6pt}%
    {\rule[-\textheight/2]{1ex}{\textheight}}
  }{\textheight}%
}{0.5ex}}%
\stackon[1pt]{#1}{\tmpbox}%
}
\newcommand{\pullback}[4]{{#1 \leftindex_{#2}{\times}_{#3} #4}}
\newcommand{\R}[1]{R_{\mathbb{E}}}
\renewcommand{\P}[1]{P_{\mathbb{E}}}
\newcommand{\la}{\langle}
\newcommand{\ra}{\rangle}
\newcommand{\un}{\underline}
\usepackage{color}

\def\pullback{
 \ar@{-}[]+R+<4pt,-3pt>;[]+RD+<4pt,-6pt>%
 \ar@{-}[]+D+<3pt,-6pt>;[]+RD+<4pt,-6pt>}

\DeclareSymbolFont{alephbet}{HE8}{frank}{m}{n}
\SetSymbolFont{alephbet}{bold}{HE8}{frank}{b}{n}
\DeclareMathSymbol{\samech}{\mathord}{alephbet}{"F1}
\DeclareMathSymbol{\mem}{\mathord}{alephbet}{"EE}

\setlist[itemize]{label=$-$} 

\usepackage{chngcntr}
\counterwithin*{equation}{section}

\def\pullback{
 \ar@{-}[]+R+<4pt,-3pt>;[]+RD+<4pt,-6pt>%
 \ar@{-}[]+D+<3pt,-6pt>;[]+RD+<4pt,-6pt>}

\begin{document}

\title[A direction functor approach to the cohomology of small categories]{A direction functor approach to the cohomology of small categories}

\author[S.~Ambra]{Stefano Ambra}
\author[A.~Duvieusart]{Arnaud Duvieusart}
\author[A.~Montoli]{Andrea Montoli}

\email{stefano.ambra@unimi.it}
\email{arnaud.duvieusart@uclouvain.be}
\email{andrea.montoli@unimi.it}

\address[Stefano Ambra]{Dipartimento di Matematica ``Federigo Enriques'', Universit\`{a} degli Studi di Milano, Via Saldini 50, 20133 Milano, Italy}
\address[Arnaud Duvieusart]{Institut de Recherche en Mathématique et Physique, Université catholique de Louvain, Chemin du Cyclotron 2, 1348 Louvain-la-Neuve, Belgium}
\address[Andrea Montoli]{Dipartimento di Matematica ``Federigo Enriques'', Universit\`{a} degli Studi di Milano, Via Saldini 50, 20133 Milano, Italy}

\begin{abstract}
We show how the direction functors can be used to develop a cohomology theory for Barr-exact and $S$-Mal'tsev categories, where $S$ is a suitable class of split epimorphisms with a fixed section. Using the fact that, for any set $B$, the category of small categories with $B$ as set of object is $S$-Mal'tsev with respect to the class of Schreier points, we show that the cohomology theory of small categories arising from the direction functors coincides with the one introduced by Hoff and Golasi\'{n}ski.
\end{abstract}


\maketitle

\section{Introduction}
The cohomology theory for groups, introduced in \cite{EML}, can be described in terms of suitable $n$-extensions, as shown in \cite{holt, huebschmann}. A categorical explanation of this fact can be obtained by means of the so-called \emph{direction functors}, first introduced in \cite{baer-sums} in dimension one, and then extended to higher dimensions in \cite{bourn-rodelo}. The first-dimensional direction functor associates with any object with global support and equipped with an internal Mal'tsev operation, in a Barr-exact \cite{barr} category $\mathcal{C},$ an internal group in $\mathcal{C}.$ If $\mathcal{C}$ is a Mal'tsev \cite{carboni} category, this group is necessarily abelian. When $\mathcal{C}$ is a slice category $\mathbf{Gp}/G,$ where $\mathbf{Gp}$ is the category of groups, the objects with global support and admitting an internal Mal'tsev operation are precisely the surjective group homomorphisms with abelian kernel, and the direction functor associates with any such morphism $f$ the split extension corresponding to the action induced by $f$ on its kernel. The fibres of the direction functors inherit from the objects of the codomain category a canonical abelian group structure, which coincides with the Eilenberg-Mac Lane second cohomology group. Higher dimensional direction functors, where extensions with abelian kernels are replaced by suitable internal groupoids, equip their fibres with abelian group structures that are isomorphic to the higher cohomology groups. Applying the same strategy to other algebraic categories, like associative algebras or Lie algebras, gives a description of other known cohomology theories.

On the other hand, in recent years there has been interest in the study of monoids (and, more generally, of monoid-like structures) from a categorical point of view. Indeed, it has been observed that the category of monoids, which is not a Mal'tsev category, retains most of the good algebraic and homological properties of the category of groups, when the attention is restricted to a suitable class of points (by point we mean a split extension with a fixed section), called Schreier points \cite{MMS13, schreier_book}. This led to the notions of $S$-protomodular \cite{S-protomodular} and $S$-Mal'tsev \cite{mal'tsev-reflection} category, with respect to a pullback-stable class of points. Examples of such categories, beside monoids, are \emph{monoids with operations} \cite{MMS13} (which include semirings, semilattices, distributive lattices with a bottom element or a top one...), quandles (that are algebraic structures used in the classification of knots), and small categories with a fixed set of objects.

The aim of this paper is to show how the direction functors can be considered in the context of Barr-exact and $S$-Mal'tsev categories, in such a way as to yield a cohomology theory for all the above structures. Moreover, by applying these functors to the case of the category $\mathbf{Cat}_B$ of small categories with fixed set of objects $B,$ we obtain a categorical description of the cohomology theory for small categories introduced in \cite{cocat} and in \cite{golasinski}.

The paper is organized as follows. In Section \ref{section direction functors} we recap from \cite{baer-sums, bourn-rodelo} the general constrution and the main properties of the direction functors. In Section \ref{section S-protomodularity} we recall the definitions of $S$-protomodular and $S$-Mal'tsev categories, and study in detail the example which is more relevant for us, namely the one of the category $\mathbf{Cat}_B.$ In Section \ref{section second cohomology CatB} we apply the first-dimensional direction functor to $\mathbf{Cat}_B,$ and provide an explicit description of the resulting cohomology groups. Eventually, in Section \ref{section higher cohomology} we do the same for the higher dimensional direction functors, giving a description of cohomology groups, in every dimension, by means of suitable extensions.

\section{Direction functors and Cohomology} \label{section direction functors}
\label{sec:direction}
\subsection{Internal Mal'tsev algebras and naturally Mal'tsev categories}
Let $\mathcal{C}$ be a category with finite products. Recall that an \emph{internal Mal'tsev algebra} (cf.~\cite{smith}) in $\mathcal{C}$ is a pair $(X,p_X),$ where $X$ is an object of $\mathcal{C}$ and $p_X\colon X^3=X\times X\times X\rightarrow X$ is an internal ternary operation satisfying the Mal'tsev identities $p_X(x,x,y)=y$ and $p_X(x,y,y)=x.$ We say that $p_X$ is \emph{associative} when the equation $p_X\big(x,y,p_X(z,u,v)\big)=p_X\big(p_X(x,y,z),u,v\big)$ is also satisfied, and in this case we call $(X,p_X)$ an (internal) associative Mal'tsev algebra. Such objects, together with the morphisms $f\colon(X,p_X)\rightarrow(Y,p_Y),$ where $f\colon X\rightarrow Y$ is a morphism in $\mathcal{C}$ such that the square
\[
\xymatrix{
{X^3} \ar[d]_-{f^3=f\times f\times f} \ar[r]^-{p_X} &X \ar[d]^-f \\
{Y^3} \ar[r]_-{p_Y} &Y
}
\]
commutes, make for a category which we denote by $\mathrm{Mal}(\mathcal{C}).$

If $(X,p_X)$ is an associative Mal'tsev algebra and a morphism $x_0\colon1\rightarrow X$ exists in $\mathcal{C},$ $1$ being a terminal object, the binary operation $x\cdot y=p_X(x,x_0,y)$ endows $X$ with the structure of an internal group in $\mathcal{C},$ whose neutral element is $x_0$ and where $x^{-1}=p_X(x_0,x,x_0);$ this group is abelian as soon as the Mal'tsev operation $p_X$ is \emph{commutative}, i.e. when $p_X(x,y,z)=p_X(z,y,x)$ holds. Conversely, any internal group $(G,\cdot,e)$ in $\mathcal{C}$ determines a canonical associative Mal'tsev operation $p_G(x,y,z)=x\cdot y^{-1}\cdot z$ on $G,$ which is commutative if and only if the group is abelian. Thus, the associative Mal'tsev algebras in $\mathcal{C}$ play the role of non-pointed internal groups and the associative and commutative Mal'tsev algebras play the role of non-pointed internal abelian groups.

It is proven in \cite{baer-sums} (see also \cite{johnstone}) that a Mal'tsev operation $p_X$ is both associative and commutative if and only if it commutes with itself, i.e. if and only if $p_X$ is a morphism of Mal'tsev algebras:
\begin{equation*}
\begin{split}
p_X&\big(p_X(x_1,y_1,z_1),p_X(x_2,y_2,z_2),p_X(x_3,y_3,z_3)\big)\\&=p_X\big(p_X(x_1,x_2,x_3),p_X(y_1,y_2,y_3),p_X(z_1,z_2,z_3)\big).
\end{split}
\end{equation*}
Following \cite{linton}, $p_X$ is called \emph{autonomous} in this case. We shall denote by $\mathrm{AMal}(\mathcal{C})\subseteq\mathrm{Mal}(\mathcal{C})$ the full subcategory of autonomous Mal'tsev algebras.

If $\mathcal{C}$ is a Mal'tsev \cite{carboni} category, every object $X\in\mathcal{C}$ admits at most one internal Mal'tsev operation, which is in this case necessarily autonomous (see \cite{borceux-bourn}, Propositions $2.3.3$ and $2.3.7$), so that $\mathrm{Mal}(\mathcal{C})=\mathrm{AMal}(\mathcal{C}).$ In this context, internal Mal'tsev algebras are also called \emph{abelian objects}, see \cite{bourn-abelian}.

Recall also from \cite{johnstone} that the category $\mathcal{C}$ is \emph{naturally Mal'tsev} if there exists a natural transformation $p\colon()^3\xlongrightarrow{\cdot} Id_{\mathcal{C}}$ from the functor $X\mapsto X^3,$ $f\mapsto f^3,$ to the identity functor of $\mathcal{C},$ such that for all $X\in\mathcal{C}$ the component $p_X$ of $p$ in $X$ is a Mal'tsev operation. The naturality of $p$ then forces $p_X$ to be autonomous, so that $\mathcal{C}$ is naturally Mal'tsev if and only if $\mathcal{C}=\mathrm{AMal}(\mathcal{C}).$ Every naturally Mal'tsev category is a Mal'tsev category (see \cite{bourn96}, Proposition 17).

Much as the autonomous Mal'tsev algebras are the non-pointed analogues of internal abelian groups, naturally Mal'tsev categories provide for a non-pointed analogue of additive categories, since one can prove that a category with finite products $\mathcal{C}$ is additive if and only if it is naturally Mal'tsev and pointed (\cite{johnstone}, p.~255).
\begin{remarks}
\label{ex:naturally_maltsev}
\begin{enumerate}
\item Let $(G,\cdot,e)$ be an internal group in a category $\mathcal{C}$ with finite products, and suppose $e^\prime\colon1\rightarrow G$ is some other morphism. Then, if $\cdot^\prime$ is defined by $x\cdot^\prime y=p_G(x,e^\prime,y)$ (with the above notation), the groups $(G,\cdot,e)$ and $(G,\cdot^\prime,e^\prime)$ are isomorphic in the category $\mathrm{Gp}(\mathcal{C})$ of internal groups in $\mathcal{C}$ via $(G,\cdot,e)\rightarrow(G,\cdot^\prime,e^\prime),$ $x\mapsto p_G(e^\prime,e,x).$

\item If $\mathcal{D}$ is a Mal'tsev category, for any $Y\in\mathcal{D}$ the slice category $\mathcal{D}/Y$ is again Mal'tsev (\cite{borceux-bourn}, Example $2.2.14$), so that $\mathcal{C}=\mathrm{Mal}(\mathcal{D}/Y)$ is naturally Mal'tsev: in view of the study of cohomology by means of the so-called \emph{direction functors} (\cite{baer-sums, bourn-rodelo} - see below), particularly interesting instances of this situation occur when:
\begin{itemize}
\item $\mathcal{D}$ is additive, in which case $\mathcal{D}/Y=\mathrm{Mal}(\mathcal{D}/Y)=\mathrm{AMal}(\mathcal{D}/Y)$ is itself naturally Mal'tsev (\cite{borceux-bourn}, Corollary $2.4.11$ - see also \cite{bourn-rodelo});
\item $\mathcal{D}$ is one of the categories $\mathbf{Gp}$ of groups, $\mathbf{GpTop}$ of topological groups, $\mathbf{GpHaus}$ of Hausdorff groups, $\mathbf{Lie}_R$ of Lie algebras over a (commutative) ring $R$: an object in $\mathrm{Mal}(\mathcal{D}/Y)$ is in this case a group homomorphism with codomain $Y$ and abelian kernel when $\mathcal{D}=\mathbf{Gp},$ a continuous group homomorphism with codomain $Y$ and abelian kernel when $\mathcal{D}=\mathbf{GpTop}$ or $\mathcal{D}=\mathbf{GpHaus},$ a Lie-homomorphism with codomain $Y$ and kernel with trivial Lie bracket when $\mathcal{D}=\mathbf{Lie}_R$ (see \cite{bourn-rodelo});
\end{itemize}
\item If $\mathcal{D}$ is an $S$-Mal'tsev category, for any $Y\in\mathcal{D}$ the full subcategory of $S$-special morphisms $Sl(\mathcal{D}/Y)\subset \mathcal{D}/Y$ is a Mal'tsev category \cite{schreier_book, S-protomodular}, so that $\mathcal{C}=\mathrm{Mal}(Sl(\mathcal{D}/Y))$ is naturally Mal'tsev: this is the case that we shall study in the present work.
\end{enumerate}
\end{remarks}
\subsection{The direction of an associative Mal'tsev algebra and of an $n$-groupoid}
Suppose now that $\mathcal{C}$ is (finitely complete and) Barr-exact \cite{barr}, and fix an associative Mal'tsev algebra $(X,p_X).$ Then, by setting $(x,t)\sim(y,z)$ if and only if $t=p_X(x,y,z),$ one defines an internal equivalence relation
\begin{equation}
\label{eqn:Chasles}
\xymatrixcolsep{2.5pc}
\xymatrix{
{X\times X\times X} \ar@<.5ex>[r]^-{\langle \pi_1,p_X\rangle} \ar@<-.5ex>[r]_-{\pi_{2,3}} &{X\times X}
}
\end{equation}
on $X\times X,$  where $\pi_1$ and $\pi_{2,3}$ denote the projections on the first factor and on the second and third factors, respectively.

Following \cite{baer-sums}, we shall call \eqref{eqn:Chasles} the \emph{Chasles relation} associated with $p_X,$ though a similar construction was considered in \cite{kock} under the name of \emph{geometric} relation.

By the Barr-exactness of $\mathcal{C},$ the relation \eqref{eqn:Chasles} is effective and admits a quotient $\big(d(X,p_X),\rho\big)$:
\begin{equation*}
\xymatrixrowsep{2.5pc}
\xymatrixcolsep{2.5pc}
\xymatrix{
{X\times X\times X} \ar[d]_-{p_{1,2}} \ar@{}[rd]|{(a)} \ar@<.5ex>[r]^-{\langle \pi_1,p_X\rangle} \ar@<-.5ex>[r]_-{\pi_{2,3}} &{X\times X} \ar[d]^-{\pi_1} \ar@{->>}[r]^-{\rho} \ar@{}[rd]|{(b)} &{d(X,p_X)} \ar[d] \\
{X\times X} \ar@<.5ex>[r]^-{\pi_1} \ar@<-.5ex>[r]_-{\pi_2} &{X} \ar[r]_-{\tau_X} &{1.}
}
\end{equation*}
Since the two commutative squares $(a)$ are pullbacks, the so-called Barr-Kock theorem (\cite{barr}, Example III.6.10) guarantees that the square $(b)$ is also a pullback, meaning that $X\times X\cong X\times d(X,p_X).$

The main point, now, is that when the terminal map $\tau_X$ is a regular epimorphism - a fact which we express by saying that $X$ has \emph{global support} (\cite{tower}, p.~144) - the quotient $d(X,p_X)$ is an \emph{internal group} in $\mathcal{C}.$ In this respect, we take a diversion from the path outlined in \cite{baer-sums}, choosing to stress the point of view of groups as \emph{pointed} associative Mal'tsev algebras, and to derive the group structure on $d(X,p_X)$ from the following:
\begin{proposition}
If $(X,p_X)\in\mathrm{Mal}(\mathcal{C})$ (resp., $(X,p_X)\in\mathrm{AMal}(\mathcal{C})$), the quotient $d(X,p_X)$ inherits an associative (resp., autonomous) Mal'tsev algebra structure.
\end{proposition}
\begin{proof}
By \cite{barr}, Proposition I.2.15, $\rho^3$ is a coequalizer of $\langle \pi_1,p_X \rangle^3$ and ${\pi_{2,3}}^3.$ Consider the morphism
\[
\varphi\colon X^6\rightarrow X^2, \ (a,b,x,y,u,v)\mapsto\Big(a,p_X\big(b,y,p_X(x,u,v)\big)\Big).
\]
Using the equality
\[
p_X\big(x,p_X(y,z,t),u\big)=p_X\big(p_X(x,t,z),y,u\big),
\]
valid for any associative Mal'tsev operation $p_X$ (see \cite{baer-sums}, Corollary 2), one proves that the composition $\rho\varphi\colon X^6\rightarrow d(X,p_X)$ coequalizes $\langle \pi_1,p_X \rangle^3$ and ${\pi_{2,3}}^3,$ so that a unique
\begin{equation}
\label{eqn:induced_maltsev_operation}
p_{d(X,p_X)}\colon d(X,p_X)^3\rightarrow d(X,p_X)
\end{equation}
is induced, satisfying $\rho\varphi=p_{d(X,p_X)}\rho^3.$ A routine computation then shows that \eqref{eqn:induced_maltsev_operation} is an associative Mal'tsev operation on $d(X,p_X),$ which is also commutative if so is $p_X.$
\end{proof}

\begin{corollary}
If $(X,p_X)\in\mathrm{Mal}(\mathcal{C})$ and $X$ has global support, $d(X,p_X)$ admits a group structure in $\mathcal{C},$ which is abelian if $p_X$ is autonomous.
\end{corollary}
Indeed, when $\tau_X$ is a regular epimorphism, the universal property of $\tau_X$ as a coequalizer of the indiscrete relation $\xymatrix{{X\times X} \ar@<.5ex>[r]^-{\pi_1} \ar@<-.5ex>[r]_-{\pi_2} &{X}}$ guarantees the existence of a map $u_X$ making the upward square in the right-hand side of
\begin{equation*}
\xymatrixcolsep{3.5pc}
\xymatrix{
{X\times X\times X} \ar@<-.5ex>[d]_-{\pi_{1,2}} \ar@<.5ex>[r]^-{\langle \pi_1,p_X\rangle} \ar@<-.5ex>[r]_-{\pi_{2,3}} &{X\times X} \ar@<-.5ex>[d]_-{\pi_1} \ar@{->>}[r]^-{\rho} &{d(X,p_X)} \ar@<-.5ex>[d] \\
{X\times X} \ar@<-.5ex>[u]_-{\la 1_{X\times X},\pi_2 \ra} \ar@<.5ex>[r]^-{\pi_1} \ar@<-.5ex>[r]_-{\pi_2} &X \ar@<-.5ex>[u]_-{\Delta_X=\langle 1_X,1_X\rangle} \ar@{->>}[r]_-{\tau_X} &{1} \ar@{.>}@<-.5ex>[u]_-{u_X}
}
\end{equation*}
commute. Then $d(X,p_X)$ is a pointed associative Mal'tsev algebra, and hence an internal group in $\mathcal{C},$ with multiplication $\rho(a,b)\cdot \rho(u,v)=\rho\big(a,p_X(b,u,v)\big)$ induced by \eqref{eqn:induced_maltsev_operation}.

\begin{remark}
Having global support is a necessary condition on $X$ for $d(X,p_X)$ to be a group, for if the latter holds, then the terminal map $d(X,p_X)\rightarrow 1$ is a split epimorphism, and thus, in the commutative diagram
\[
\xymatrix{
{X\times X} \ar[d]_-{\pi_1} \ar@{->>}[r]^-{\rho} &{d(X,p_X)} \ar@{->>}[d] \\
{X} \ar[r]_-{\tau_X} &{1,}
}
\]
$\tau_X$ is a regular epimorphism.
\end{remark}

Again by the universal property of the coequalizer, the map $(X,p_X)\mapsto d(X,p_X)$ extends to the morphisms in $\mathrm{Mal}(\mathcal{C})$ and defines functors $d$
\begin{equation}
\begin{aligned}
\label{eqn:dir_funs}
\xymatrix{
{\mathrm{Mal}(\mathcal{C}_g)} \ar[r]^-d &{\mathrm{Gp}(\mathcal{C})} \\
{\mathrm{AMal}(\mathcal{C}_g)} \ar[r]_-d \ar@{^{(}->}[u] &{\mathrm{Ab}(\mathcal{C}),} \ar@{^{(}->}[u]
}
\end{aligned}
\end{equation}
where $\mathcal{C}_g\subseteq\mathcal{C}$ is the full subcategory of objects with global support (also denoted by $\mathcal{C}_{\sharp}$ in \cite{bourn-rodelo}) and $\mathrm{Gp}(\mathcal{C}),$ $\mathrm{Ab}(\mathcal{C})$ are the categories of internal groups and internal abelian groups in $\mathcal{C},$ respectively.

The group $d(X,p_X),$ defined up to isomorphisms, is called the \emph{direction} of $(X,p_X)\in\mathrm{Mal}(\mathcal{C}_g),$ and consequently, the name \emph{direction functor} is employed for the functors \eqref{eqn:dir_funs}. They were first introduced in \cite{baer-sums}. (The terminology here refers to the link between the theory of associative Mal'tsev algebras in $\mathcal{C}=\mathbf{Set}$ and plane affine geometry, which is studied extensively in \cite{bourn-affine}.)

As the Mal'tsev operation $p_X$ is associative, when $X\in\mathcal{C}_g$ the isomorphism $X\times d(X,p_X)\cong X\times X$ results in a simply transitive \emph{group} action of $d(X,p_X)$ on $X,$ making $X$ into a $d(X,p_X)$-\emph{torsor} \cite{torsors} (or a \emph{principal (right)} $d(X,p_X)$-\emph{object}, in the terminology of \cite{barr}, Definition IV.3.1). Conversely, one can prove that any simply transitive group action of an internal group $G$ on $X$ determines a Mal'tsev operation on $X$ whose direction is $G$: see \cite{baer-sums}, Corollary 4.

Thus, if $\mathcal{C}$ is a Mal'tsev category (so that the possible Mal'tsev operation $p_X$ on $X$ is unique and autonomous), we conclude that $d(X,p_X)=d(X)$ is uniquely determined, up to isomorphisms in $\mathrm{Ab}(\mathcal{C}),$ by its acting on $X$ with a simply transitive group action.
\begin{remark}
The full power of a Barr-exact category is actually not quite necessary to construct the functors $d,$ since Barr-exactness is only used to guarantee the effectiveness of the Chasles relations \eqref{eqn:Chasles} (and the good behaviour of their quotients): for this to hold, however, it would be enough for $\mathcal{C}$ to be \emph{effectively regular} (\cite{bourn-rodelo}, Definition 1.5, also named \emph{efficiently regular} in \cite{efficiently}). Thus, the scope of the above construction covers all the examples listed in Remark \ref{ex:naturally_maltsev}$(2).$ See \cite{bourn-rodelo} for the general theory in this setting.
\end{remark}

If $\mathcal{C}$ is a Mal'tsev category, we already observed that $\mathrm{Mal}(\mathcal{C})=\mathrm{AMal}(\mathcal{C})$ is naturally Mal'tsev, and it is not difficult to prove that $\mathrm{Mal}(\mathcal{C}_g)=\big(\mathrm{Mal}(\mathcal{C})\big)_g$: since our main interest in what follows lies in categories that are indeed Barr-exact and Mal'tsev, for the rest of this section we shall consider directly a Barr-exact and naturally Mal'tsev category $\mathcal{C}$ and write the direction functor of $\mathcal{C}$ as

\begin{equation}
\label{eqn:direction_functor}
d\colon\mathcal{C}_g \longrightarrow \mathrm{Ab}(\mathcal{C}).
\end{equation}

The main features of $d$ are collected in the following:
\begin{theorem}[{\cite{baer-sums}, Proposition 6, Theorem 7}]
\label{thm:properties_of_d}
The direction functor \eqref{eqn:direction_functor} is a cofibration, and it preserves finite products, regular epimorphisms and all existing pullbacks. Moreover, it is conservative, so that as a consequence of the previous properties:
\begin{itemize}
\item every map in $\mathcal{C}_g$ is cocartesian for $d;$
\item every fibre of $d$ is a groupoid;
\item $d$ reflects pullbacks.
\end{itemize}
\end{theorem}
\begin{remark}
The functor $d$ is actually a \emph{pseudo}-cofibration, meaning that for every morphism $\alpha\colon d(X)\rightarrow A$ in $\mathbf{Ab}(\mathcal{C})$ there exists a morphism $f\colon X\rightarrow Y$ in $\mathcal{C}_g$ which is cocartesian over $\alpha$ only up to an isomorphism $d(Y)\cong A$ (see~\cite{baer-sums}, Theorem 7). This fact does not undermine the properties that are discussed below.
\end{remark}
Next, we extend $d$ to the category $n$-$\mathrm{Gpd}(\mathcal{C})$ of $n$-groupoids and internal $n$-functors in $\mathcal{C},$ always with $\mathcal{C}$ Barr-exact and naturally Mal'tsev.

Recall that, for every natural number $n,$ an $n$\emph{-groupoid} in $\mathcal{C}$ is defined inductively as follows:
\begin{itemize}
\item for $n=0,$ a $0$-groupoid is just an object $X\in\mathcal{C};$
\item for $n=1,$ a $1$-groupoid is an ordinary groupoid: since by assumption $\mathcal{C}$ is naturally Mal'tsev, this amounts to having a reflexive graph $\xymatrix{{\underline{X}_1:X_1} \ar@<.9ex>[r]^-{x_0} \ar@<-.9ex>[r]_-{x_1} &{X_0} \ar[l]|(.39){s_0}}$ (\emph{Lawvere condition}: see \cite{johnstone} or \cite{bourn96}), and there results a functor $()_0\colon1\text{-}\mathrm{Gpd}(\mathcal{C})=\mathrm{Gpd}(\mathcal{C})\rightarrow\mathcal{C}$ associating with every groupoid $\underline{X}_1$ its object of objects $X_0;$
\item for $n\geq 2,$ an $n$-groupoid in $\mathcal{C}$ is a groupoid in the fibre $()^{-1}_{n-2}(\underline{X}_{n-2})$ of the functor $()_{n-2}\colon(n-1)\text{-}\mathrm{Gpd}(\mathcal{C})\rightarrow(n-2)\text{-}\mathrm{Gpd}(\mathcal{C})$ for some $(n-2)$-groupoid $\underline{X}_{n-2},$ and we have a functor $()_{n-1}\colon n\text{-}\mathrm{Gpd}(\mathcal{C})\rightarrow(n-1)\text{-}\mathrm{Gpd}(\mathcal{C})$ which ``cuts off'' the $n$-th component.
\end{itemize}
Thus, explicitly, an $n$-groupoid in $\mathcal{C}$ is a (triple) sequence
\begin{equation*}
\xymatrix{
{\underline{X}_n:X_n} \ar@<.9ex>[r]^-{x_0} \ar@<-.9ex>[r]_-{x_1} &{X_{n-1}} \ar[l]|(.44){s_0} \ar@<.9ex>[r]^-{x_0} \ar@<-.9ex>[r]_-{x_1} &{\dots} \ar[l]|(.43){s_0} \ar@<.9ex>[r]^-{x_0} \ar@<-.9ex>[r]_-{x_1} &{X_0} \ar[l]|{s_0}
}
\end{equation*}
of objects and morphisms of $\mathcal{C}$ satisfying at each level the equations $x_0s_0=x_1s_0=1,$ $x_0x_1=x_0x_0$ and $x_1x_0=x_1x_1,$ and a morphism $\underline{f}_n\colon\underline{X}_n\rightarrow\underline{Y}_n$ of $n$-groupoids (which is called an \emph{internal} $n$\emph{-functor}) is a tuple $(f_n,\dots,f_0)$ of morphisms such that the diagram
\begin{equation*}
\xymatrix{
{X_n} \ar[d]_-{f_n} \ar@<.9ex>[r]^-{x_0} \ar@<-.9ex>[r]_-{x_1} &{X_{n-1}} \ar[d]_-{f_{n-1}} \ar[l]|(.55){s_0} \ar@<.9ex>[r]^-{x_0} \ar@<-.9ex>[r]_-{x_1} &{\dots} \ar[l]|(.43){s_0} \ar@<.9ex>[r]^-{x_0} \ar@<-.9ex>[r]_-{x_1} &{X_0} \ar[d]^-{f_0} \ar[l]|{s_0} \\
{Y_n} \ar@<.9ex>[r]^-{y_0} \ar@<-.9ex>[r]_-{y_1} &{Y_{n-1}} \ar[l]|(.55){z_0} \ar@<.9ex>[r]^-{y_0} \ar@<-.9ex>[r]_-{y_1} &{\dots} \ar[l]|(.43){z_0} \ar@<.9ex>[r]^-{y_0} \ar@<-.9ex>[r]_-{y_1} &{Y_0} \ar[l]|{z_0}
}
\end{equation*}
is commutative, in the usual sense.

Limits in $n\text{-}\mathrm{Gpd}(\mathcal{C})$ are computed level-wise as limits in $\mathcal{C},$ and similarly regular epimorphisms are level-wise regular epimorphisms in $\mathcal{C}.$ The category $n\text{-}\mathrm{Gpd}(\mathcal{C})$ is still naturally Mal'tsev and Barr-exact (or effectively regular), when so is $\mathcal{C}$ (cf. \cite{bourn-rodelo} and \cite{gran}).

Now, for $n\geq 1,$ the functor $()_{n-1}$ has a right adjoint $\nabla_n\colon(n-1)\text{-}\mathrm{Gpd}(\mathcal{C})\rightarrow n\text{-}\mathrm{Gpd}(\mathcal{C})$ given by $\xymatrix{
{\nabla_n(\underline{X}_{n-1}):P_{n-1}} \ar@<.9ex>[r]^-{\pi_0} \ar@<-.9ex>[r]_-{\pi_1} &{X_{n-1}} \ar[l]|(.33){\sigma_0} \ar@<.9ex>[r]^-{x_0} \ar@<-.9ex>[r]_-{x_1} &{\dots} \ar[l]|(.43){s_0} \ar@<.9ex>[r]^-{x_0} \ar@<-.9ex>[r]_-{x_1} &{X_0,} \ar[l]|{s_0}
}$ where:
\begin{itemize}
\item $P_0=X_0\times X_0;$
\item $(P_{n-1},\pi_0,\pi_1)$ is a kernel pair of the induced morphism $(x_0,x_1)\colon X_{n-1}\longrightarrow P_{n-2}$ and $\sigma_0=(1,1).$ (Hence the choice of $P_{n-1}$ for ``parallel'' $(n-1)$-cells.)
\end{itemize}
The unit of the adjunction $()_{n-1}\dashv\nabla_n$ is given for every $\underline{X}_n\in n\text{-}\mathrm{Gpd}(\mathcal{C})$ by
\begin{equation}
\begin{aligned}
\label{eqn:eta_n}
{
\xymatrixcolsep{0.1pc}
\xymatrix{
{\underline{X}_n} \ar[d]_-{\underline{\eta}_n} &{:}\\
{\nabla_n(\underline{X}_{n-1})} &{:}
}
}
{
\xymatrix{
{X_n} \ar[d]_-{(x_0,x_1)} \ar@<.9ex>[r]^-{x_0} \ar@<-.9ex>[r]_-{x_1} &{X_{n-1}} \ar@{=}[d] \ar[l]|(.55){s_0} \ar@<.9ex>[r]^-{x_0} \ar@<-.9ex>[r]_-{x_1} &{\dots} \ar[l]|(.43){s_0} \ar@<.9ex>[r]^-{x_0} \ar@<-.9ex>[r]_-{x_1} &{X_0} \ar@{=}[d] \ar[l]|{s_0} \\
{P_{n-1}} \ar@<.9ex>[r]^-{\pi_0} \ar@<-.9ex>[r]_-{\pi_1} &{X_{n-1}} \ar[l]|{\sigma_0} \ar@<.9ex>[r]^-{x_0} \ar@<-.9ex>[r]_-{x_1} &{\dots} \ar[l]|(.43){s_0} \ar@<.9ex>[r]^-{x_0} \ar@<-.9ex>[r]_-{x_1} &{X_0,} \ar[l]|{s_0}
}
}
\end{aligned}
\end{equation}
and it is then a consequence of the following proposition that, for every $n\geq 1,$ the functor $()_{n-1}$ is a fibration (in the sense of Grothendieck \cite[VI]{sga1}):
\begin{proposition}
Let a functor $F\colon\mathcal{U}\rightarrow\mathcal{V}$ have a right adjoint $G.$ Then a morphism $h\colon X\rightarrow Y$ in $\mathcal{U}$ is cartesian for $F$ if and only if the naturality square
\[
\xymatrix{
X \ar[d]_-h \ar[r]^-{\eta_X} &{GF(X)} \ar[d]^-{GF(h)} \\
Y \ar[r]_-{\eta_Y} &{GF(Y)}
}
\]
of the unit of the adjunction $F\dashv G$ is a pullback.
\end{proposition}
(See for example \cite{BCGS}, Proposition 36, for a proof.)

A level-wise application of the functor $d$ to $n$-groupoids and $n$-functors yields a functor $\underline{d}_n\colon n\text{-}\mathrm{Gpd}(\mathcal{C}_g)\rightarrow n\text{-}\mathrm{Gpd}\big(\mathrm{Ab}(\mathcal{C})\big),$ which is nothing but the direction functor $d\colon\mathcal{D}_g\rightarrow\mathrm{Ab}(\mathcal{D})$ of the category $\mathcal{D}=n\text{-}\mathrm{Gpd}(\mathcal{C}),$ since it follows from the nature of regular epimorphisms and limits in $n\text{-}\mathrm{Gpd}(\mathcal{C})$ that $n\text{-}\mathrm{Gpd}(\mathcal{C}_g)=\big(n\text{-}\mathrm{Gpd}(\mathcal{C})\big)_g$ and $n\text{-}\mathrm{Gpd}\big(\mathrm{Ab}(\mathcal{C})\big)=\mathrm{Ab}\big(n\text{-}\mathrm{Gpd}(\mathcal{C})\big).$

Then the kernel of $\underline{d}_n(\underline{\eta}_n)$ in $\mathrm{Ab}\big(n\text{-}\mathrm{Gpd}(\mathcal{C})\big)=n\text{-}\mathrm{Gpd}\big(\mathrm{Ab}(\mathcal{C})\big)$ is completely determined by an internal abelian group $A$ in $\mathcal{C}$
\begin{equation*}
\begin{aligned}
{
\xymatrixcolsep{0.1pc}
\xymatrix{
{K_n(A)=Ker\big(\underline{d}_n(\underline{\eta}_n)\big)} \ar@{>->}[d] &{:} \\
{\underline{d}_n(\underline{X}_n}) \ar[d]_-{\underline{d}_n(\underline{\eta}_n)} &{:}\\
{\underline{d}_n(\nabla_n\underline{X}_{n-1})} &{:}
}
}
{
\xymatrixcolsep{1pc}
\xymatrix{
{A} \ar@<.9ex>[r] \ar@<-.9ex>[r] \ar@{>->}[d] &{1}\ar[l] \ar[d] \ar@{=}[r] &{\dots} \ar@{=}[r] &{1} \ar[d] \\
{d(X_n)} \ar[d]_-{d(x_0,x_1)} \ar@<.9ex>[r] \ar@<-.9ex>[r] &{d(X_{n-1})} \ar@{=}[d] \ar[l] \ar@<.9ex>[r] \ar@<-.9ex>[r] &{\dots} \ar[l] \ar@<.9ex>[r] \ar@<-.9ex>[r] &{d(X_0)} \ar@{=}[d] \ar[l] \\
{d(P_{n-1})} \ar@<.9ex>[r] \ar@<-.9ex>[r] &{d(X_{n-1})} \ar[l] \ar@<.9ex>[r] \ar@<-.9ex>[r] &{\dots} \ar[l] \ar@<.9ex>[r] \ar@<-.9ex>[r] &{d(X_0),} \ar[l]
}
}
\end{aligned}
\end{equation*}
and we call this $A=Ker\big(d(x_0,x_1)\colon d(X_n)\rightarrow d(P_{n-1})\big)$ the $n$\emph{-direction} of the $n$-groupoid $\underline{X}_n.$ The map $\underline{X}_n\mapsto A$ extends to a functor $d_n\colon n\text{-}\mathrm{Gpd}(\mathcal{C})_g\rightarrow\mathrm{Ab}(\mathcal{C})$ (introduced in \cite{bourn-rodelo}). We also set $d_n=d$ when $n=0.$
\subsection{Properties of the direction functors and connection with Cohomology}
As it is shown in \cite{baer-sums, bourn-rodelo}, the functors $d_n$ (with $n\geq0$) allow for a unified description of many cohomological theories of interest. The key point in this respect is that, by the properties of (appropriate restrictions of) the $d_n$'s, the connected components of the fibres $d_n^{-1}(A)$ inherit a canonical abelian group structure, which in familiar cases (typically $\mathcal{C}=\mathrm{Mal}(\mathcal{D}/Y),$ as in the Examples \ref{ex:classical_cohomologies} below) coincide with the usual cohomology groups $H^n(Y,A)$ of $Y$ with coefficients in $A.$

Our goal is to apply this theory to $\mathcal{C}=\mathrm{Mal}(\mathcal{E}),$ where $\mathcal{E}=Sl(\mathcal{D}/Y)\subseteq \mathcal{D}/Y$ is the full subcategory of $S$-special morphisms with fixed codomain $Y$ in an appropriate $S$-protomodular category $\mathcal{D},$ to show that other existent cohomological theories, such as the ones described in \cite{golasinski, hoff} for (small) categories, can also be framed in these terms.

The properties of the functors $d_n$'s which ensure that the (possibly large) set $\pi_0\big(d_n^{-1}(A)\big)$ of connected components of $d_n^{-1}(A)$ is endowed with a natural structure of an abelian group are collected in the following general result:
\begin{proposition}[{\cite{bourn-rodelo}, Proposition 2.11 and Theorem 2.12} - cf. \cite{baer-sums}, Theorem 9]
\label{prop:fibres_abelian_groups}
Let $\mathcal{E}$ be a category with finite products and $\mathcal{A}$ a finitely complete additive category. Consider a functor $F\colon\mathcal{E}\rightarrow\mathcal{A}$ such that:
\begin{enumerate}
\item $F$ is a (pseudo-)cofibration;
\item $F$ preserves finite products;
\item the cocartesian maps are stable under finite products.
\end{enumerate}
Then, for any $A\in\mathcal{A},$ the fibre $F^{-1}(A)$ admits a canonical symmetric monoidal structure $\big(F^{-1}(A),\otimes,I,\dots\big)$ where $X\otimes Y$ is defined by considering $X\times Y\rightarrow X\otimes Y$ cocartesian over the sum $+:A\times A\rightarrow A$ and $1\rightarrow I$ cocartesian over $0\rightarrow A.$

Moreover, if
\begin{enumerate}\setcounter{enumi}{3}
\item $\mathcal{E}$ has kernel pairs and $F$ preserves them, and
\item the subdiagonal of any cocartesian map (i.e. the reflexivity morphism of its kernel relation) is cocartesian (see \cite{bourn-rodelo}, Theorem $2.12$),
\end{enumerate}
then $\pi_0\big(F^{-1}(A)\big)$ admits a canonical structure of an abelian group.
\end{proposition}
(Recall that $X,Y\in F^{-1}(A)$ are in the same connected component if there exist objects $X_1,\dots,X_{k-1},Y_1,\dots,Y_k$ and maps
\begin{equation}
\begin{aligned}
\label{eqn:zig-zag}
{
\xymatrixcolsep{1.5pc}
\xymatrix{
 &{Y_1} \ar[dl] \ar[dr] &\ &{Y_2} \ar[dl] \ar[dr] &\ &\ \ar[dl] \\
X &\ &{X_1} &\ &{X_2} &\
}
}
\xymatrix{
{\dots}\\
{\dots}
}
{
\xymatrixcolsep{1.5pc}
\xymatrix{
\ar[rd]  &\ &{Y_k} \ar[dl] \ar[dr] &\ \\
&{X_{k-1}} &\ &Y
}
}
\end{aligned}
\end{equation}
in the fibre $F^{-1}(A).$)

By Theorem \ref{thm:properties_of_d}, the conditions of Proposition \ref{prop:fibres_abelian_groups} are surely met by $d_n$ when $n=0.$ For the same conditions to apply also when $n\geq 1,$ we must actually consider the restriction of $d_n\colon n\text{-}\mathrm{Gpd}(\mathcal{C}_g)\rightarrow\mathrm{Ab}(\mathcal{C})$ to the so-called \emph{aspherical} groupoids:
\begin{definition}[\cite{aspherical, bourn-rodelo}]
Let $\mathcal{C}$ be as before.
\begin{itemize}
\item A $0$-groupoid $X_0\in\mathcal{C}$ is aspherical when it has global support, i.e. when the terminal map $X_0\rightarrow 1$ is a regular epimorphism.
\item For $n\geq 1,$ an $n$-groupoid $\underline{X}_n$ is aspherical when $\underline{X}_{n-1}=(\underline{X}_n)_{n-1}$ is aspherical and $\underline{\eta}_n\colon\underline{X}_n\rightarrow\nabla_n(\underline{X}_{n-1})$ (cf.~\eqref{eqn:eta_n}) is a regular epimorphism in $n\text{-}\mathrm{Gpd}(\mathcal{C}).$
\end{itemize}
\end{definition}
Denote by $n\text{-}\mathrm{Asp}(\mathcal{C})\subseteq n\text{-}\mathrm{Gpd}(\mathcal{C})_g$ the full subcategory of aspherical $n$-groupoids in $\mathcal{C}.$ Following \cite{bourn-rodelo}, the resulting functor
\begin{equation}
\label{eqn:d_n}
d_n\colon n\text{-}\mathrm{Asp}(\mathcal{C})\longrightarrow\mathrm{Ab}(\mathcal{C})
\end{equation}
is called the $n$\emph{-dimensional direction functor} of $\mathcal{C}.$

Then one can prove:
\begin{theorem}[\cite{bourn-rodelo}, Theorems 3.6 and 4.6]
For every $n\geq 1,$ the $n$-dimensional direction functor $d_n\colon n\text{-}\mathrm{Asp}(\mathcal{C})\rightarrow\mathrm{Ab}(\mathcal{C})$ is a pseudo-cofibration, and it preserves finite products and all existing pullbacks. A map in $n\text{-}\mathrm{Asp}(\mathcal{C})$ is cocartesian for $d_n$ if and only if it is $()_{n-1}$-invertible, and it is cartesian for $()_{n-1}$ if and only if it is $d_n$-invertible: thus, the cocartesian maps for $d_n$ are stable under finite products. Finally, the subdiagonals of $d_n$-cocartesian maps are cocartesian, and $d_n$ reflects isomorphisms when it is restricted to cocartesian maps (by \cite{bourn-rodelo}, Theorem 3.6(7)).
\end{theorem}
\begin{remarks}
\begin{enumerate}
\item The fibres $d_n^{-1}(A)$ are no longer groupoids when $n\geq 1$ ($d_n$ is no longer conservative), but one can show that the length of any zig-zag \eqref{eqn:zig-zag} for two connected objects can be reduced to $k=1.$ This is a rather technical point, since $n\text{-}\mathrm{Asp}(\mathcal{C})$ does not have all pullbacks in general, see \cite{bourn-rodelo}.
\item It is also shown in \cite{bourn-rodelo} that an object $\underline{X}_n\in d_n^{-1}(A)$ is in the connected component of $0=K_n(A):\xymatrix{{A} \ar@<.9ex>[r] \ar@<-.9ex>[r] &{1}\ar[l] \ar@{=}[r] &{\dots} \ar@{=}[r] &{1}}$ if and only if there is a map $\nabla_n(\underline{Z}_{n-1})\rightarrow\underline{X}_n$ in $n\text{-}\mathrm{Asp}(\mathcal{C})$ for some aspherical $(n-1)$-groupoid $\underline{Z}_{n-1}$ (\cite{bourn-rodelo}, Propositions 3.9 and 4.9).
\end{enumerate}
\end{remarks}
As a consequence, by Proposition \ref{prop:fibres_abelian_groups} we have indeed for the functors $d_n\colon n\text{-}\mathrm{Asp}(\mathcal{C})\rightarrow\mathrm{Ab}(\mathcal{C})$:
\begin{corollary}
\label{cor:connected_components_group}
For every $A\in\mathrm{Ab}(\mathcal{C})$ and every $n\geq 0,$ $\pi_0\big(d_n^{-1}(A)\big)=H^{n+1}_{\mathcal{C}}(A)$ is canonically an abelian group.
\end{corollary}

It is then the main result of \cite{bourn-rodelo} that every short exact sequence $0\rightarrow A\rightarrow B \rightarrow C \rightarrow 0$ in $\mathrm{Ab}(\mathcal{C})$ induces a long exact sequence of abelian groups and group homomorphisms
\begin{equation}
\label{eqn:long_exact_seq}
\cdots \rightarrow H^{n}_{\mathcal{C}}(B) \rightarrow H^{n}_{\mathcal{C}}(C)\rightarrow H^{n+1}_{\mathcal{C}}(A) \rightarrow H^{n+1}_{\mathcal{C}}(B) \rightarrow \cdots
\end{equation}
which in the classical cases (see the examples below) coincides with the usual long exact cohomology sequence, and which is now obtained in a non-necessarily abelian context and without the assumption of having enough projective objects. The sequence \eqref{eqn:long_exact_seq} extends the six-term exact sequence of \cite{barr}, Chapter V.
\begin{examples}
\label{ex:classical_cohomologies}
\begin{enumerate}
\item If $\mathcal{A}$ is an abelian category and $Y\in\mathcal{A},$ the group $H^n_{\mathcal{A}/Y}(A)$ is nothing but the classical Yoneda's $\mathrm{Ext}^n_{\mathcal{A}}(Y,A)$ group, see \cite{bourn-rodelo}, Examples 2.15(1), 3.8 and 4.8. The canonical group operation on $H^n_{\mathcal{A}/Y}(A)$ induced by the symmetric monoidal structure on $d_n^{-1}(A)$ in this case coincides precisely with the Baer sum of $n$-extensions.
\item Similarly, the application of this theory to $\mathcal{C}=\mathrm{Mal}(\mathcal{D}/Y)$ for $\mathcal{D}=\mathbf{Gp}$ or $\mathcal{D}=\mathbf{Lie}_R$ yields the classical cohomology theories for groups and Lie algebras over a ring $R,$ see \cite{bourn-rodelo}. Other relevant instances arise by considering the categories $\mathcal{D}=\mathbf{GpTop}$ of topological groups or $\mathcal{D}=\mathbf{GpHaus}$ of Hausdorff groups, see again \cite{bourn-rodelo}.
\end{enumerate}
\end{examples}
\section{$S$-protomodular categories and Schreier points} \label{section S-protomodularity}
\subsection{$S$-protomodular and $S$-Mal'tsev categories}
Recall that a \emph{point} in a category $\mathcal{C}$ is a pair of morphisms $\xymatrix{A \ar@<.5ex>[r]^-f &B \ar@<.5ex>[l]^-s }$ such that $s$ is a fixed section of $f$ (i.e. $fs=1_B$). Together with the morphisms $g,h$ in $\mathcal{C}$ such that the diagrams
\[
\xymatrix{
{A^\prime} \ar[r]^-g \ar@<-.5ex>[d]_-{f^\prime} &A \ar@<-.5ex>[d]_-f \\
{B^\prime} \ar[r]_h \ar@<-.5ex>[u]_-{s^\prime} &B \ar@<-.5ex>[u]_-s
}
\]
commute both downwards and upwards, the points in $\mathcal{C}$ form a category which we denote by $Pt(\mathcal{C}).$ The category $Pt(\mathcal{C})$ is finitely complete if so is $\mathcal{C},$ and regular epimorphisms in $Pt(\mathcal{C})$ are level-wise regular epimorphisms in $\mathcal{C}.$

A point $(f,s)$ is \emph{strong} \cite{S-protomodular} if, for every downward pullback
\begin{equation}
\begin{aligned}
\label{eqn:point_pullback}
\xymatrix{
P \pullback \ar[r]^-u \ar@<-.5ex>[d]_-{f^\prime} &A \ar@<-.5ex>[d]_-f \\
C \ar[r]_v \ar@<-.5ex>[u]_-{s^\prime} &B \ar@<-.5ex>[u]_-s
}
\end{aligned}
\end{equation}
of $(f,s)$ along a morphism $v$ in $\mathcal{C},$ the pair $\{u,s\}$ is jointly extremal epimorphic. (Earlier variants of the name, in the pointed context: strongly split epimorphism \cite{bourn-monad}; regular point \cite{MMS14}.)

It is a well-known result that a finitely complete category $\mathcal{C}$ is protomodular in the sense of \cite{bourn-proto} if and only if every point in $\mathcal{C}$ is strong (see for example \cite{borceux-bourn}, Lemma 3.1.22). Thus the following notion, which was first considered in the pointed case \cite{schreier_book, S-protomodular}, provides a \emph{relative} version of protomodularity:
\begin{definition}[\cite{mal'tsev-reflection}, Definition 4.3]
A finitely complete category $\mathcal{C}$ is $S$\emph{-protomodular} with respect to a pullback-stable class of points $S$ in $\mathcal{C}$ if:
\begin{enumerate}
\item every point in $S$ is strong;
\item the full subcategory $SPt(\mathcal{C})\subseteq Pt(\mathcal{C})$ whose objects are the points in $S$ is closed under finite limits in $Pt(\mathcal{C}).$ (A class $S$ satisfying this condition is called \emph{point-congruous} in \cite{mal'tsev-reflection, partial, 3x3}.)
\end{enumerate}
\end{definition}
The stability of $S$ under pullbacks (in $\mathcal{C}$) means that, in any pullback diagram \eqref{eqn:point_pullback} in which $(f,s)$ is in $S,$ the induced point $(f^\prime,s^\prime)$ is also in $S.$

Condition $(2)$ is a technical requirement that is sometimes omitted from the general definition of $S$-protomodularity (see for example \cite{partial}, Definition 8.5, or \cite{3x3}, Definition 5.3), but which will be essential for our purposes. For example, it implies that the trivial point $\xymatrix{1 \ar@<.5ex>[r] &{1,} \ar@<.5ex>[l]}$ and thus (by pullback stability) any isomorphism $\xymatrix{X \ar@<.5ex>[r]^-f &{Y,} \ar@<.5ex>[l]^-{f^{-1}}}$ belongs to the class $S.$ Moreover, if $\mathcal{C}$ is pointed, it guarantees that the Split Short Five Lemma for points in the class $S$ holds in $\mathcal{C}$: see \cite{S-protomodular}, Proposition 3.2, or \cite{schreier_book}, Proposition 8.1.2.

Now, every protomodular category is a Mal'tsev category (\cite{bourn96}, Proposition 17), and, by a classical result of Bourn, a finitely complete category $\mathcal{C}$ is a Mal'tsev category if and only if every fibre $Pt_B(\mathcal{C})$ of the codomain functor $Pt(\mathcal{C})\rightarrow\mathcal{C},$ $\xymatrix{{(A} \ar@<.5ex>[r]^-f &{B)} \ar@<.5ex>[l]^-s }\mapsto B,$ is \emph{unital} \cite{bourn96}, a fact which amounts to the following property: For every downward pullback in $\mathcal{C}$
\begin{equation}
\begin{aligned}
\label{eqn:point-product}
\xymatrix{
P \pullback \ar@<.5ex>[r]^-{g^\prime} \ar@<-.5ex>[d]_-{f^\prime} &A \ar@<.5ex>[l]^{t^\prime} \ar@<-.5ex>[d]_-f \\
C \ar@<.5ex>[r]^-g \ar@<-.5ex>[u]_-{s^\prime} &B \ar@<.5ex>[l]^t \ar@<-.5ex>[u]_-s
}
\end{aligned}
\end{equation}
of a point $(f,s)$ along a point $(g,t),$ the induced morphisms $s^\prime$ and $t^\prime$ are jointly extremal epimorphic (see \cite{bourn96}, Proposition 10).

In the same spirit of $S$-protomodularity, then, the following definition provides an $S$-version of the Mal'tsev property:
\begin{definition}[\cite{mal'tsev-reflection}, Definition 4.3]
A finitely complete category $\mathcal{C}$ is an $S$\emph{-Mal'tsev} category with respect to a pullback-stable class of points $S$ in $\mathcal{C}$ if:
\begin{enumerate}
\item for every pullback diagram \eqref{eqn:point-product} with $(f,s)\in S,$ the pair $\{s^\prime,t^\prime\}$ is jointly extremal epimorphic;
\item $SPt(\mathcal{C})$ is closed under finite limits in $Pt(\mathcal{C}).$
\end{enumerate}
\end{definition}
Unsurprisingly, we have:
\begin{proposition}[\cite{mal'tsev-reflection}, Theorem 4.1]
Any $S$-protomodular category is an $S$-Mal'tsev category.
\end{proposition}
It is a consequence of the pullback stability of $S$ that two isomorphic points in $Pt(\mathcal{C})$ are either both in $S$ or both outside $S,$ so that the ensuing definition is well posed:
\begin{definition}[\cite{mal'tsev-reflection}, Definitions 4.4 and 4.5]
\label{def:S-special}
Let $\mathcal{C}$ be an $S$-Mal'tsev category.
\begin{itemize}
\item An internal equivalence relation
\begin{equation*}
\xymatrix{
{R} \ar@<.9ex>[r]^-{r_0} \ar@<-.9ex>[r]_-{r_1} &X \ar[l]|{s_0}
}
\end{equation*}
in $\mathcal{C}$ is an $S$\emph{-equivalence relation} if $(r_0,s_0)\in S$ (if and only if $(r_1,s_0)\in S$).

\item A morphism $f\colon X\rightarrow Y$ in $\mathcal{C}$ is $S$\emph{-special} if its kernel relation $\xymatrix{
{\mathrm{Eq}(f)} \ar@<.9ex>[r]^-{f_0} \ar@<-.9ex>[r]_-{f_1} &X \ar[l]|(.40){\Delta_X}
}$ is an $S$-equivalence relation.

(From now on, when $f\colon X\rightarrow Y$ is a map of sets, we shall tacitly use the notation $\mathrm{Eq}(f)$ to denote the standard realization $\{(x,y)\in X\times X:f(x)=f(y)\}$ of its kernel pair.)
\end{itemize}
\end{definition}
\begin{remarks}
\label{rmks:s-special}
If $\mathcal{C}$ is an $S$-Mal'tsev category:
\begin{enumerate}
\item the class of $S$-special morphisms is pullback-stable in $\mathcal{C},$ because so is $S;$
\item since $S$ contains all isomorphisms, every monomorphism in $\mathcal{C}$ is $S$-special;
\item it is generally false that a composition of $S$-special morphisms is $S$-special (a simple counterexample is given by the composition $\xymatrix{{\mathbb{N}} \ar@{^(->}[r] &{\mathbb{Z}} \ar[r] &0}$ in $\mathbf{Mon}$ with respect to the class $S$ of Schreier points, see below), but one can prove that if, for a composition $f=hg,$ both $f$ and $h$ are $S$-special, then $g$ is also $S$-special (see \cite{partial}, Proposition 6.6);
\item it is generally false that if a point $\xymatrix{A \ar@<.5ex>[r]^-f &B \ar@<.5ex>[l]^-s}$ is in $S,$ then $f$ is $S$-special (again in the case of $\mathbf{Mon}$ with the class of Schreier points, a counterexample is given by $\xymatrix{\mathbb{N} \ar@<.5ex>[r] &{0),} \ar@<.5ex>[l]}$ but it is true that if $(f,s)$ is a point in $\mathcal{C}$ and $f$ is $S$-special, then $(f,s)$ belongs to $S$ (see \cite{S-protomodular}, Proposition 6.6).
\end{enumerate}
\end{remarks}
For a fixed object $Y\in\mathcal{C},$ we shall denote by $Sl(\mathcal{C}/Y)$ the full subcategory of the slice category $\mathcal{C}/Y$ whose objects are the $S$-special morphisms with codomain $Y.$ Limits in $Sl(\mathcal{C}/Y)$ are computed as in $\mathcal{C}/Y.$

Then one can prove:
\begin{proposition}[\cite{partial}, Theorem 6.7]
\label{prop:S-special-Maltsev}
If $\mathcal{C}$ is an $S$-Mal'tsev category, for every $Y\in\mathcal{C}$ the category $Sl(\mathcal{C}/Y)$ is a Mal'tsev category.
\end{proposition}
Regarding Barr-exactness, the next result is proved in \cite{3x3}, Proposition 4.9:
\begin{proposition}
\label{prop:special_exact}
Let $\mathcal{C}$ be an $S$-Mal'tsev category with respect to a class of points $S$ which moreover satisfies the following condition $(\mathscr{P})$: For every regular epimorphism $(g,h)$ in $Pt(\mathcal{C}),$ as in the diagram
\begin{equation}
\begin{aligned}
\label{eqn:2-regular}
\xymatrix{
{\mathrm{Eq}(g)} \ar@<.9ex>[r]^-{g_0} \ar@<-.9ex>[r]_-{g_1} \ar@<-.5ex>[d]_-{\widehat{f}}  &{A} \ar[l] \ar@{->>}[r]^-g \ar@<-.5ex>[d]_-{f} &{A^\prime} \ar@<-.5ex>[d]_-{f^\prime} \\
{\mathrm{Eq}(h)} \ar@<.9ex>[r]^-{h_0} \ar@<-.9ex>[r]_-{h_1} \ar@<-.5ex>[u]_-{\widehat{s}}  &{B} \ar[l] \ar@{->>}[r]_h \ar@<-.5ex>[u]_-s &{B^\prime,} \ar@<-.5ex>[u]_-{s^\prime}
}
\end{aligned}
\end{equation}
whose domain $(f,s)$ is in $S,$ the codomain $(f^\prime,s^\prime)$ is in $S$ as soon as the point $(\widehat{f},\widehat{s})$ induced between the kernel pairs is in $S.$

Then, if $\mathcal{C}$ is Barr-exact, for every $Y\in\mathcal{C}$ the category $Sl(\mathcal{C}/Y)$ is also Barr-exact.
\end{proposition}
A class $S$ satisfying $(\mathscr{P})$ is called \emph{2-regular} in \cite{partial, 3x3}. The notion was introduced in \cite{partial}, Definition 7.12.

As it is explained in \cite{schreier_book, S-protomodular}, motivation for the study of $S$-protomodularity came from the observation that the non-abelian, non-protomodular category $\mathbf{Mon}$ of (small) monoids and monoid homomorphisms shares some homological features with the non-abelian, protomodular category $\mathbf{Gp}$ of (small) groups and group homomorphisms, when the attention is restricted to a specific class of points in $\mathbf{Mon},$ namely the class of Schreier points \cite{MMS13, schreier_book}.

Recall (from \cite{MMS13}, Definition 2.6) that a point $(f,s)$
\begin{equation}
\label{eqn:schreier_point}
\xymatrix{K \ar@{>->}[r]^-k &X \ar@<.5ex>[r]^-f &M \ar@<.5ex>[l]^-s }
\end{equation}
in $\mathbf{Mon}$ (with $(K,k)$ a fixed kernel of $f$) is a \emph{Schreier point} if every $x\in X$ admits a decomposition $x=k(a)\cdot sf(x)$ for a unique $a\in K.$

Then one can prove that the classical correspondence between group actions and points in $\mathbf{Gp}$ extends to a correspondence between monoid actions and Schreier points in $\mathbf{Mon}$ (see \cite{MMS13} and \cite{schreier_book}, Proposition 5.2.2), and that the Split Short Five Lemma holds in $\mathbf{Mon}$ limited to the class $S$ of Schreier points (\cite{schreier_book}, Corollary 2.3.8). Indeed, the category of monoids, and more in general any category of \emph{monoids with operations} (a context which also includes groups, rings, associative algebras, Lie algebras, commutative monoids, semirings, join-semilattices with a bottom element, distributive lattices with a top or a bottom element) is $S$-protomodular with respect to the class $S$ of Schreier points, see \cite{S-protomodular}. The category $\mathbf{Qnd}$ of \emph{quandles} (introduced in \cite{joyce, matveev} for applications in Knot Theory), together with the class $S$ of \emph{acupuncturing} points, is an example of an $S$-Mal'tsev category which is not $S$-protomodular, see \cite{partial}.

Now, the Schreier points \eqref{eqn:schreier_point} in $\mathbf{Mon}$ are a particular case of the broader notion of Schreier points in $\mathbf{Cat}_B,$ which we explore in the next section and which will provide us with the appropriate environment for the study of the cohomology theories \cite{golasinski, hoff}.
\subsection{Schreier points in $\mathbf{Cat}_B$}
Denote by $\mathbf{Cat}$ the category of small categories and functors and by $\mathbf{Cat}_B,$ for every set $B,$ the fibre $\mathbf{Cat}_B=()_0^{-1}(B)$ of the set-of-objects functor
\begin{equation*}
()_0\colon\mathbf{Cat}\rightarrow\mathbf{Set}, \ \xymatrix{{(\underline{X}:X_2\cong X\times_B X} \ar[r]^-{m_{\un{X}}} &{X} \ar@<.9ex>[r]^-{x_0} \ar@<-.9ex>[r]_-{x_1} &{B)} \ar[l]|-{s_0}}\mapsto B.
\end{equation*}

Then, if $B=1$ is a terminal object in $\mathbf{Set},$ $\mathbf{Cat}_B=\mathbf{Cat}_1$ is equivalent to $\mathbf{Mon}.$

For any set $B,$ the functor $()_0$ is an exact fibration, meaning that every fibre $\mathbf{Cat}_B$ is Barr-exact and every change-of-base functor is exact (see \cite{tour}). It is worth noting that, while $\mathbf{Cat}_B$ is Barr-exact, the whole category $\mathbf{Cat}$ is not even regular (a detailed account on epimorphisms in $\mathbf{Cat}$ can be found for example in \cite{foundations}, Chapter VIII).

For a small category $\xymatrix{{\underline{X}:X_2} \ar[r]^-{m_{\un{X}}} &{X} \ar@<.9ex>[r]^-{x_0} \ar@<-.9ex>[r]_-{x_1} &{B,} \ar[l]|-{s_0}}$ we shall denote by $s_0(b)=1_b,$ for every $b\in B,$ the identity arrows, and by $y\cdot x\colon a\rightarrow c$ the composition $m_{\un{X}}(x,y)$ of $x\colon a\rightarrow b$ and $y\colon b\rightarrow c.$
\begin{definition}
\label{def:schreier_point_cat}
A point $(\underline{f},\underline{s})$ in $\mathbf{Cat}_B$
\begin{equation}
\begin{aligned}
\label{eqn:sch-pt-cat}
{
\xymatrix{
{\underline{X}} \ar@<-.5ex>[d]_-{\underline{f}=(f,1_B)}\\
{\underline{Y}} \ar@<-.5ex>[u]_-{\underline{s}=(s,1_B)}
}
}
{\xymatrix{: \\ :}}
{
\xymatrixcolsep{3.5pc}
\xymatrix{
{X_2} \ar[r]^-{m_{\un{X}}} \ar@{.>}@<-.5ex>[d] &{X} \ar@<-.5ex>[d]_-{f} \ar@<.9ex>[r]^-{x_0} \ar@<-.9ex>[r]_-{x_1} &{B} \ar@{=}[d] \ar[l]|{s_0} \\
{Y_2} \ar[r]_-{m_{\un{Y}}} \ar@{.>}@<-.5ex>[u] &{Y} \ar@<-.5ex>[u]_-{s} \ar@<.9ex>[r]^-{y_0} \ar@<-.9ex>[r]_-{y_1} &{B} \ar[l]|{t_0}
}
}
\end{aligned}
\end{equation}
is a \emph{Schreier point} if for every $y\in Y$ the arrow $s(y)\in X$ is cocartesian for $\underline{f}.$
\end{definition}
This notion comes from \cite{mal'tsev-reflection}, where it appears under the name of \emph{points with cofibrant splittings}.

The following characterization holds:
\begin{proposition}
\label{prop:char_schreier_points}
A point \eqref{eqn:sch-pt-cat} in $\mathbf{Cat}_B$ is a Schreier point if and only if for every $x\in X,$ with $x\colon a\rightarrow b,$ there exists a unique $k\in f^{-1}(1_b)$ such that $x=k\cdot sf(x).$
\end{proposition}
\begin{proof}
Suppose that \eqref{eqn:sch-pt-cat} meets the conditions of the proposition and consider arrows $y\colon a\rightarrow b$ in $Y$ and $x\colon a\rightarrow c$ in $X$ such that $f(x)=z\cdot y$ for some $z\colon b\rightarrow c$ in $Y.$ Then, if $k\in f^{-1}(1_c)$ is the unique map such that $x=k\cdot sf(x),$ the map $k\cdot s(z)\colon b\rightarrow c$ satisfies $f\big(k\cdot s(z)\big)=z$ and $k\cdot s(z)\cdot s(y)=k\cdot s(z\cdot y)=k\cdot sf(x)=x$
\begin{equation*}
\centering
\xymatrixrowsep{0.5pc}
\xymatrix{
&a \ar[ld]_-{sf(x)} \ar[r]^-{s(y)} \ar[dd]_-x &b \ar@{.>}[ldd]^-{k\cdot s(z)} \\
c \ar[rd]_-k &\ &\ \\
&c  &\
}
\xymatrixrowsep{2pc}
\xymatrix{
&\ \ar@{--}[d] \\
&\
}
\xymatrix{
&a  \ar[d]_-{f(x)}  \ar[r]^-{y} &b \ar[dl]^-{z} \\
&{c.} &\
}
\end{equation*}
If $x=w\cdot s(y)$ for some other $w\colon b\rightarrow c$ such that $f(w)=z,$ by factorizing $w=k^\prime\cdot sf(w)=k^\prime\cdot s(z)$ for a unique $k^\prime\in f^{-1}(1_c)$ we have $x=w\cdot s(y)=k^\prime\cdot s(z\cdot y)=k^\prime\cdot sf(x)$: then $k^\prime=k$ by the uniqueness of $k,$ and we conclude that $w=k\cdot s(z).$ The converse is immediate.
\end{proof}
It is now clear that, for $B=1,$ a Schreier point in $\mathbf{Cat}_1\cong\mathbf{Mon}$ in the sense of Definition \ref{def:schreier_point_cat} is precisely a Schreier point \eqref{eqn:schreier_point} of monoids in the sense of \cite{MMS13}.

One may further think of Schreier points as functors to $\mathbf{Mon},$ thanks to a variation of the classical Grothendieck construction (\cite{sga1}, Section VI.8). Indeed, observe first that if $\un{f}=(f,1_B)\colon\un{X}\rightarrow\un{Y}$ is a morphism in $\mathbf{Cat}_B,$ then for all $b\in B$ the sets $f^{-1}(1_b)$ are monoids under the composition of $\un{X}$ (the neutral element being $1_b\in X$). Now, let \eqref{eqn:sch-pt-cat} be a Schreier point in $\mathbf{Cat}_B.$ If $y\colon a\rightarrow b$ is a morphism in $Y$ and $k\in f^{-1}(1_a),$ by Proposition \ref{prop:char_schreier_points} there exists a unique $^yk\in f^{-1}(1_b)$ such that the square
\begin{equation*}
\xymatrix{
a \ar[d]_-{s(y)} \ar[r]^-k &a \ar[d]^-{s(y)} \\
b \ar[r]_-{^yk} &b
}
\end{equation*}
commutes. Then, if other morphisms $k^\prime\in f^{-1}(1_a)$ and $z\in Y$ are given, with $z\colon b\rightarrow c,$ the outer rectangles in the diagrams
\begin{equation*}
\begin{aligned}
\xymatrix{
a \ar[d]_-{s(y)} \ar[r]^-k &a \ar[d]_-{s(y)} \ar[r]^-{k^\prime} &a \ar[d]^-{s(y)} \\
b \ar[r]_-{^yk} &b \ar[r]_-{^y(k^\prime)} &{b,}
}
\
\xymatrix{
a \ar[d]_-{s(y)} \ar[r]^-k &a \ar[d]^-{s(y)} \\
b \ar[r]_-{^yk}  \ar[d]_-{s(z)} &b \ar[d]^-{s(z)} \\
c \ar[r]_-{^z(^yk)} &c
}
\end{aligned}
\end{equation*}
are also commutative, meaning that ${^y(k^\prime)}\cdot{}^y(k)={^y(k^\prime\cdot k)}$ and $^z(^yk)={^{(z\cdot y)}k},$ and we end up with a functor
\begin{equation}
\label{eqn:functor_sch_point}
\underline{Y} \rightarrow \mathbf{Mon}, \ \  (y\colon a\rightarrow b) \mapsto \big({^y()}\colon f^{-1}(1_a)\rightarrow f^{-1}(1_b), \ k\mapsto {^yk}\big).
\end{equation}

Conversely, if $\underline{Y}\in\mathbf{Cat}_B$ and $F\colon\underline{Y}\longrightarrow \mathbf{Mon}$ is a functor, define an internal category $\xymatrix{{\underline{X}:X_2} \ar[r] &{X} \ar@<.9ex>[r]^-{x_0} \ar@<-.9ex>[r]_-{x_1} &{B} \ar[l]|-{s_0}}$ by declaring as morphisms $(a\rightarrow b)\in X$ the couples $(\nu, y)$ with $y\colon a\rightarrow b$ in $Y$ and $\nu\in F(b),$ with identity arrows $s_0(b)=(e_{F(b)},1_b)$ and composition
\[
(\mu\in F(c),z\colon b\rightarrow c)\cdot(\nu\in F(b),y\colon a\rightarrow b)=\big(\mu\cdot F(z)(\nu),z\cdot y\big),
\]
where $e_{F(b)}$ is the neutral element of the monoid $F(b),$ $\mu\cdot F(z)(\nu)$ denotes the product in $F(c),$ and $z\cdot y$ is the composition in $\underline{Y}.$ Set also $x_0(\nu,y)=y_0(y)$ and $x_1(\nu,y)=y_1(y).$ (If $F$ is a functor to $\mathbf{Ab},$ the category $\un{X}$ is called the \emph{crossed product} of $F$ and $\un{Y},$ in the terminology of \cite{cocat}: the notion goes back to \cite{ehresmann}.)

Then clearly $\xymatrix{X \ar@<.5ex>[r]^-f &{Y,} \ar@<.5ex>[l]^-s }$ with $f(\nu,y)=y$ and $s(y\colon a\rightarrow b)=(e_{F(b)},y),$ yields a point $\xymatrix{{\underline{X}} \ar@<.5ex>[r]^-{\underline{f}} &{\underline{Y}} \ar@<.5ex>[l]^-{\underline{s}} }$ in $\mathbf{Cat}_B,$ which is a Schreier point by Proposition \ref{prop:char_schreier_points}, because every $(\nu,y\colon a\rightarrow b)\in X$ admits a decomposition $(\nu,y)=(\nu,1_b)\cdot(e_{F(b)},y)=(\nu,1_b)\cdot sf(\nu,y)$ for a unique $(\nu,1_b)\in f^{-1}(1_b).$

It is not difficult to see that this establishes an equivalence:
\begin{proposition}[cf. \cite{mal'tsev-reflection}, Lemma 3.1]
\label{prop:schreier_functors}
The full subcategory $SPt_{\underline{Y}}(\mathbf{Cat}_B)\subseteq Pt_{\underline{Y}}(\mathbf{Cat}_B)$ of Schreier points in $\mathbf{Cat}_B$ with codomain $\underline{Y}$ is equivalent to the functor category $\mathrm{Fun}(\underline{Y},\mathbf{Mon}).$
\end{proposition}
We already mentioned that the relevance of Schreier points, in the case $B=1,$ comes from the fact that there is an equivalence between Schreier points in $\mathbf{Mon}$ and monoid actions. It is shown in \cite{a1} that one can introduce notions of action and semidirect product also in the context of $\mathbf{Cat}_B,$ for an arbitrary set $B,$ generalizing the classical one for monoids, and study the relation between actions and Schreier points in $\mathbf{Cat}_B$ accordingly.

Indeed, fix a small category $\xymatrix{{\underline{X}:X_2} \ar[r]^-{m_{\un{X}}} &{X} \ar@<.9ex>[r]^-{x_0} \ar@<-.9ex>[r]_-{x_1} &{B,} \ar[l]|-{s_0}}$ and consider for any $a,b\in B$ the set $\un{X}(a,b)$ of all arrows $x\in X$ with domain $x_0(x)=a$ and codomain $x_1(x)=b.$ Then, for every $b\in B,$ the set $\un{X}(b,b)$ is a monoid under the composition of $\un{X},$ with neutral element $1_b=s_0(b),$ and we can define
\begin{equation*}
E_{\un{X}}(a,b)=\mathbf{Mon}\big(\un{X}(a,a),\un{X}(b,b)\big)
\end{equation*}
and a small category
\begin{equation*}
\xymatrix{E({\underline{X}):E_2} \ar[r]^-{m_E} &{E(X)} \ar@<.9ex>[r]^-{d_E} \ar@<-.9ex>[r]_-{c_E} &{B,} \ar[l]|(.40){u_E}}
\end{equation*}
whose arrows $f\in E(X)$ with domain $d_E(f)=a$ and codomain $c_E(f)=b$ are given by $E_{\un{X}}(a,b).$ Of course, units in $E(\un{X})$ are defined by $u_E(b)=id_{\un{X}(b,b)}\in E_{\un{X}}(b,b),$ and composition is given by the composition of morphisms in $\mathbf{Mon}.$

Observe that for all $a,b\in B$ there is always at least one morphism $a\rightarrow b$ in $E(\un{X}),$ given by the trivial morphism $\un{X}(a,a)\rightarrow\un{X}(b,b),$ $x\mapsto 1_b$ (so that $E(\un{X})$ is never \emph{totally disconnected} if $|B|\geq 2,$ see below). Moreover, if $B=1,$ the category $E(\un{X})$ coincides with the monoid $\big(\mathrm{End}(X),\circ,id_X\big)$ of monoid endomorphisms of $X.$

\begin{definition}[\cite{a1}]
\label{def:action_Cat_B}
If $\un{X},$ $\un{Y}\in\mathbf{Cat}_B$:
\begin{enumerate}
\item an \emph{action} of $\un{Y}$ on $\un{X}$ is a morphism $\un{\omega}\colon\un{Y}\rightarrow E(\un{X})$ in $\mathbf{Cat}_B,$ i.e. an internal functor
\begin{equation*}
\xymatrix{
{Y_2} \ar@{.>}[d] \ar[r]^-{m_{\un{Y}}} &{Y} \ar[d]_-{\omega} \ar@<.9ex>[r]^-{y_0} \ar@<-.9ex>[r]_-{y_1} &B \ar@{=}[d] \ar[l]|-{t_0} \\
{E_2} \ar[r]_-{m_E} &{E(X)} \ar@<.9ex>[r]^-{d_E} \ar@<-.9ex>[r]_-{c_E} &{B;} \ar[l]|(.40){u_E}
}
\end{equation*}
\item if $\un{\omega}\colon\un{Y}\rightarrow E(\un{X})$ is an action of $\un{Y}$ on $\un{X},$ the \emph{semidirect product} $\un{X}\rtimes_{\un{\omega}}\un{Y}$ in $\mathbf{Cat}_B$ is the small category
\begin{equation*}
\xymatrix{{\un{X}\rtimes_{\un{\omega}}\un{Y}:P_2(\un{X},\un{Y})} \ar[r]^-{\mu} &{P(\un{X},\un{Y})} \ar@<.9ex>[r]^-{d_P} \ar@<-.9ex>[r]_-{c_P} &{B} \ar[l]|(.37){u_P}}
\end{equation*}
where:
\begin{itemize}
\item an arrow $g\in P(\un{X},\un{Y})$ with domain $d_P(g)=a$ and codomain $c_P(g)=b$ is a couple $(x,y)$ with $y\in\un{Y}(a,b)$ and $x\in\un{X}(b,b);$
\item units are given by $u_P(b)=\big(s_0(b),t_0(b)\big),$ where $s_0$ is the unit map of $\un{X}$;
\item composition is defined on
\begin{equation*}
P_2(\un{X},\un{Y})\cong\big\{\big((x,y),(x',y')\big):y\in\un{Y}(a,b), \ y'\in\un{Y}(b,c)\big\}
\end{equation*}
by
\begin{equation*}
\begin{split}
(x',y')\cdot(x,y)&=\mu\big((x,y),(x',y')\big)\\
&=\Big(m_{\un{X}}\big(\omega(y')(x),x'\big),m_{\un{Y}}(y,y')\Big)\\
&=\big(x'\cdot\omega(y')(x),y'\cdot y\big)
\end{split}
\end{equation*}
(observe that indeed $\omega(y'\colon b\rightarrow c)\colon\un{X}(b,b)\rightarrow\un{X}(c,c)$).
\end{itemize}
 \end{enumerate}
\end{definition}
It is immediate that, when $B=1,$ the above notions coincide with the usual notions of a monoid action of $Y$ on $X$ (namely, a monoid homomorphism $\omega\colon Y\rightarrow\mathrm{End}(X)$) and of the semidirect product of monoids $X\rtimes_{\omega}Y.$ Moreover, observe that if $|B|\geq 2,$ the set of arrows $P(\un{X},\un{Y})$ of $\un{X}\rtimes_{\un{\omega}}\un{Y}$ is different, in general, from the set of arrows of the binary product $\un{X}\times\un{Y}$ in $\mathbf{Cat}_B$: the two coincide  when both $\un{X}$ and $\un{Y}$ are totally disconnected (i.e., there is no arrow $a\rightarrow b$ in $\un{X}$ and $\un{Y}$ if $a\neq b$), and the action $\un{\omega}$ is trivial (i.e., $\omega(y)=id_{\un{X}(b,b)}$ for all $b\in B$). Finally, observe that $\xymatrix{{\un{X}\rtimes_{\un{\omega}}\un{Y}} \ar@<.5ex>[r]^-{\un{\pi}} &{\un{Y}} \ar@<.5ex>[l]^-{\un{\sigma}}}$ is always a Schreier point in $\mathbf{Cat}_B,$ where $\pi(x,y)=y$ and $\sigma(y\colon a\rightarrow b)=\big(1_b=s_0(b),y\big),$ and that the domain $\un{X}$ of the Schreier point on $\un{Y}$ correspondig to a functor $F\colon\un{Y}\rightarrow\mathbf{Mon}$ as in Proposition \ref{prop:schreier_functors} can be now described as the semidirect product $\un{X}\cong \un{F}^+\rtimes_{\un{\omega}}\un{Y},$ where $\un{F}^+$ (in the notation of \cite{cocat, golasinski, hoff}) is the totally disconnected category on $B$ whose set of arrows is the disjoint union $\bigsqcup_{b\in B}F(b),$ and $\un{\omega}$ is the action $\omega(y\colon a\rightarrow b)=F(y)\colon F(a)\rightarrow F(b).$

Crucially, by \cite{mal'tsev-reflection}, Theorem 3.1 and Propositions 3.4, 3.6 (see also \cite{S-protomodular} for the case $B=1$), we have:
\begin{theorem}
For every set $B,$ the category $\mathbf{Cat}_B$ is $S$-protomodular with respect to the class $S$ of Schreier points.
\end{theorem}
From now on, when writing $S$ in $\mathbf{Cat}_B$ we shall always mean the class of Schreier points.

Observe that if a morphism $\underline{f}=(f,1_B)\colon\underline{X}\rightarrow\underline{Y}$ in $\mathbf{Cat}_B$ is $S$-special (Definition \ref{def:S-special}), then for every $b\in B$ the monoid $f^{-1}(1_b)\subseteq X$ is a group. Indeed, every $(x,1_b)\in\mathrm{Eq}(f)$ factorizes as $(x,1_b)=(1_b,z)\cdot(x,x)=(x,z\cdot x)$ for a unique $z\in X,$ which is thus a left inverse of $x\in f^{-1}(1_b)$ (the fact that every element is left-invertible then implies that these left inverses are also inverses on the right). It follows that the functor $\underline{X}\rightarrow \mathbf{Mon}$ associated to the Schreier point $\xymatrix{{\mathrm{Eq}(\underline{f})} \ar@<.5ex>[r]^-{\underline{f}_0} &{\underline{X}} \ar@<.5ex>[l]^-{\underline{\Delta}_X} }$ as in \eqref{eqn:functor_sch_point} is actually a functor $\underline{X}\rightarrow \mathbf{Gp}.$

Also, if $\underline{f}\colon\underline{X}\rightarrow\underline{Y}$ is $S$-special, for every $(x,y)\in\mathrm{Eq}(f)$ with $x,y\colon a\rightarrow b$ there is a unique element $q(x,y)\in f^{-1}(1_b)$ such that $q(x,y)\cdot x=y.$ This gives a function $q\colon\mathrm{Eq}(f)\rightarrow K(f)=\bigsqcup_{b\in B}f^{-1}(1_b),$ which is easily seen to satisfy the following properties:
\begin{proposition}
\label{prop:prop_retraction}
For all $k\in f^{-1}(1_b),$ $x,y\colon a\rightarrow b$ and $x^\prime,y^\prime\colon b\rightarrow c$ such that $f(x)=f(y)$ and  $f(x^\prime)=f(y^\prime)$:
\begin{enumerate}
\item $q(1_b,k)=k;$
\item $q(x,x)=1_b;$
\item $q\big((x^\prime,y^\prime)\cdot(x,y)\big)=q(x^\prime,y^\prime)\cdot q\big(x^\prime,x^\prime\cdot q(x,y)\big).$
\end{enumerate}
\end{proposition}
In the case $B=1,$ the function $q$ is usually called \emph{Schreier retraction} (see \cite{schreier_book, S-protomodular}), a name which we shall freely adopt also in the general case of $\mathbf{Cat}_B.$ Proposition \ref{prop:prop_retraction} generalizes to $\mathbf{Cat}_B$ the analogous properties of the Schreier retraction in $\mathbf{Mon},$ see \cite{schreier_book}, Proposition 2.1.5, or \cite{S-protomodular}, Proposition 4.6.

Finally, observe that by defining ${^xk}=q(x,x\cdot k)$ for all $x:a\rightarrow b$ in $X$ and $k\in f^{-1}(1_a),$ so that $x\cdot k={^xk}\cdot x,$ one obtains a functor
\begin{equation}
\label{eqn:functor_s-special}
\underline{X} \rightarrow \mathbf{Gp}, \ \  (x\colon a\rightarrow b) \mapsto \big({^x()}\colon f^{-1}(1_a)\rightarrow f^{-1}(1_b), \ k\mapsto {^xk}\big),
\end{equation}
which is isomorphic to the functor $\underline{X}\rightarrow \mathbf{Gp}$ associated to the Schreier point $\xymatrix{{\mathrm{Eq}(\underline{f})} \ar@<.5ex>[r]^-{\underline{f}_0} &{\underline{X}} \ar@<.5ex>[l]^-{\underline{\Delta}_X} }$ as in \eqref{eqn:functor_sch_point}.
\begin{remark}
\label{rmk:actions_and_schreier_points_Cat_B}
If $\un{f}\colon\un{X}\rightarrow\un{Y}$ is a morphism in $\mathbf{Cat}_B,$ the disjoint union $K(f)=\bigsqcup_{b\in B}f^{-1}(1_b)$ is a realization of the pullback $B\times_YX$ in $\mathbf{Set}$ of $f$ along the map of units $t_0$ of $\un{Y}.$ The latter extends to a pullback
\begin{equation}
\begin{aligned}
\label{eqn:K_f}
\xymatrixrowsep{1.5pc}
\xymatrixcolsep{1.5pc}
\xymatrix{
{K(\un{f})}\ar[d] \pullback \ar[r]^-{\un{j}} &{\un{X}} \ar[d]^-{\un{f}} \\
{\un{B}} \ar[r]_-{\un{t_0}} &{\un{Y}}
}
\end{aligned}
\end{equation}
in $\mathbf{Cat}_B,$ where $K(\un{f})$ is the category $\xymatrix{{K(f)} \ar@<.9ex>[r]^-{x_0j} \ar@<-.9ex>[r]_-{x_1j} &{B} \ar[l]|(.40){\beta_0}}$ (with $\beta_0=(1_B,s_0):B\rightarrow K(f)=B\times_YX$) and $\un{B}$ denotes the discrete category on $B.$ If $\un{f}$ is $S$-special, the category $K(\un{f})$ is a groupoid.

Then, if $\mathrm{Act}_{\un{Y}},$ for a fixed $\un{Y}\in\mathbf{Cat}_B,$ denotes the category of all pairs $(\un{X},\un{\omega})$ of objects $\un{X}\in\mathbf{Cat}_B$ endowed with an action $\un{\omega}\colon\un{Y}\rightarrow E(\un{X})$ (in the sense of Definition \ref{def:action_Cat_B}), whose morphisms $\un{\varphi}\colon(\un{X},\un{\omega})\rightarrow(\un{Z},\un{\vartheta})$ are the morphisms $\un{\varphi}\colon\un{X}\rightarrow\un{Z}$ in $\mathbf{Cat}_B$ such that for all $y\in\un{Y}(a,b)$ the square
\begin{equation*}
\xymatrix{
{\un{X}(a,a)} \ar[d]_-{\varphi} \ar[r]^-{\omega(y)} &{\un{X}(b,b)} \ar[d]^-{\varphi} \\
{\un{Z}(a,a)} \ar[r]_-{\vartheta(y)} &{\un{Z}(b,b)}
}
\end{equation*}
commutes, consider the functors
\begin{equation}
\label{eqn:K}
\mathcal{K}\colon SPt_{\un{Y}}(\mathbf{Cat}_B)\longrightarrow\mathrm{Act}_{\un{Y}}, \ (\xymatrix{{\underline{X}} \ar@<.5ex>[r]^-{\underline{f}} &{\underline{Y}} \ar@<.5ex>[l]^-{\underline{s}}})\mapsto (K(\un{f}),\un{\omega}),
\end{equation}
where $\omega(y\colon a\rightarrow b)\colon f^{-1}(1_a)\rightarrow f^{-1}(1_b),$ $k\mapsto {^yk}$ as in \eqref{eqn:functor_sch_point}, and
\begin{equation*}
\mathcal{P}\colon\mathrm{Act}_{\un{Y}}\longrightarrow SPt_{\un{Y}}(\mathbf{Cat}_B), \ (\un{X},\un{\omega})\mapsto(\xymatrix{{\underline{X}\rtimes_{\un{\omega}}\un{Y}} \ar@<.5ex>[r]^-{\underline{\pi}} &{\underline{Y}} \ar@<.5ex>[l]^-{\underline{\sigma}}}).
\end{equation*}
It is proven in \cite{a1} that the two are linked by an adjunction $\mathcal{K}\dashv\mathcal{P},$ which is an equivalence if and only if $|B| \leq 1.$ The functor $\mathcal{K},$ which is fully faithful, realizes $SPt_{\un{Y}}(\mathbf{Cat}_B)$ as the coreflective subcategory $(\mathrm{Act}_{\un{Y}})_{td}\subseteq\mathrm{Act}_{\un{Y}}$ of all $(\un{X},\un{\omega})$ with $\un{X}$ totally disconnected. This explains the relation between Schreier points and actions in $\mathbf{Cat}_B,$ for the general set $B.$
\end{remark}

We conclude this section by proving the following:
\begin{proposition}
Suppose that for a regular epimorphism $(\underline{g},\underline{h})$ in $Pt(\mathbf{Cat}_B)$ both its domain $(\underline{f},\underline{s})$ and the point $(\underline{\widehat{f}},\underline{\widehat{s}})$ induced between the kernel pairs, as in the diagram
\begin{equation}
\begin{aligned}
\label{eqn:sch-2-regular}
\xymatrix{
{\mathrm{Eq}(\underline{g})} \ar@<.9ex>[r]^-{\underline{g}_0} \ar@<-.9ex>[r]_-{\underline{g}_1} \ar@<-.5ex>[d]_-{\underline{\widehat{f}}}  &{\underline{X}} \ar[l] \ar@{->>}[r]^-{\underline{g}} \ar@<-.5ex>[d]_-{\underline{f}} &{\underline{X}^\prime} \ar@<-.5ex>[d]_-{{\underline{f}}^\prime} \\
{\mathrm{Eq}(\underline{h})} \ar@<.9ex>[r]^-{\underline{h}_0} \ar@<-.9ex>[r]_-{\underline{h}_1} \ar@<-.5ex>[u]_-{\underline{\widehat{s}}}  &{\underline{Y}} \ar[l] \ar@{->>}[r]_{\underline{h}} \ar@<-.5ex>[u]_-{\underline{s}} &{\underline{Y}^\prime,} \ar@<-.5ex>[u]_-{{\underline{s}}^\prime}
}
\end{aligned}
\end{equation}
are Schreier points. Then the codomain $(\underline{f}^\prime,\underline{s}^\prime)$ is also a Schreier point.
\end{proposition}
\begin{proof}
Fix a map $x^\prime\colon a\rightarrow b$ in $X^\prime.$ Since $g$ is surjective and $(\underline{f},\underline{s})$ is a Schreier point, we have $x^\prime=g(x)$ for some $x\colon a\rightarrow b$ in $X$ which factorizes as $x=k\cdot sf(x)$ for a unique $k\in f^{-1}(1_b).$ Then $x^\prime=g(x)=g(k)\cdot gsf(x)=g(k)\cdot s^\prime f^\prime(x^\prime)$ and $f^\prime g(k)=hf(k)=h(1_b)=1_b,$ so that $g(k)\in (f^\prime)^{-1}(1_b).$

To prove that $g(k)$ is unique, suppose that $x^\prime=u^\prime\cdot s^\prime f^\prime(x^\prime)$ for some other $u^\prime\in(f^\prime)^{-1}(1_b),$ with $u^\prime=g(u).$ Again, there is a unique $z\in f^{-1}(1_b)$ such that $u=z\cdot sf(u),$ and for this $z$ we have $u^\prime=g(u)=g(z)\cdot gsf(u)=g(z)\cdot s^\prime f^\prime(u^\prime)=g(z),$ so that $g\big(z\cdot sf(x)\big)=g(z)\cdot gsf(x)=u^\prime\cdot s^\prime f^\prime(x^\prime)=x^\prime=g(x).$ This means that $\big(x,z\cdot sf(x)\big)\in\mathrm{Eq}(g),$ and since $(\underline{\widehat{f}},\underline{\widehat{s}})$ is a Schreier point we must have $\big(x,z\cdot sf(x)\big)=(v,w)\cdot\big(sf(x),sf(x)\big)=\big(v\cdot sf(x),w\cdot sf(x)\big)$ for two unique elements $v,w\in f^{-1}(1_b)$ satisfying also $g(v)=g(w).$ But now uniqueness in $f^{-1}(1_b)$ entails that $v=k$ and $w=z,$ so that $u^\prime=g(z)=g(k).$
\end{proof}
(Observe that the surjectivity of $h$ plays no role in the proof, but it is nonetheless forced by the equality $f^\prime g=hf$ and the surjectivity of $f^\prime$ and $g.$)

We already remarked that the fibres $\mathbf{Cat}_B$ are Barr-exact, so that by Proposition \ref{prop:special_exact} we have:
\begin{corollary}
\label{cor:S-special-exact}
For every set $B$ and every $\underline{Y}\in\mathbf{Cat}_B,$ the category $Sl(\mathbf{Cat}_B/\underline{Y})$ of $S$-special morphisms with codomain $\underline{Y}$ is Barr-exact.
\end{corollary}
\section{The second cohomology group of a small category} \label{section second cohomology CatB}
Recall (from \cite{cocat}, p.~183) that a $\underline{Y}$\emph{-module}, for a small category $\underline{Y}$ over a set $B,$ is simply a functor $A\colon\underline{Y}\rightarrow\mathbf{Ab}.$

In this section, we shall prove that the class of extensions of a $\underline{Y}$-module $A$ by $\underline{Y},$ in the sense of \cite{cocat, hoff}, can be recovered as the fibre $d_{\un{Y}}^{-1}(A)$ of the direction functor of the category $Sl(\mathbf{Cat}_B/\underline{Y})$ of $S$-special morphisms in $\mathbf{Cat}_B$ with codomain $\underline{Y}$ ($S$ being the class of Schreier points).

The resulting abelian group structure on the connected components of $d_{\un{Y}}^{-1}(A)$ (see Corollary \ref{cor:connected_components_group}) coincides with the one given in \cite{golasinski} by the Baer sum in $\mathrm{Opext}^1(\underline{Y},A),$ thus allowing for a conceptual description of the second cohomology group $H^2(\underline{Y},A)$ in the sense of \cite{cocat, golasinski, hoff}.

\subsection{The direction functor for $S$-special morphisms in $\mathbf{Cat}_B$}
We know from Proposition \ref{prop:S-special-Maltsev} and Corollary \ref{cor:S-special-exact} that $Sl(\mathbf{Cat}_B/\underline{Y})$ is a Barr-exact and Mal'tsev category, so that every object $\underline{f}\in Sl(\mathbf{Cat}_B/\underline{Y})$ admits at most one internal Mal'tsev operation (which is necessarily autonomous). The direction functor of $\mathcal{C}=Sl(\mathbf{Cat}_B/\underline{Y})$ is then a functor
\[
d=d_{\un{Y}}\colon\mathrm{Mal}(\mathcal{C}_g)=\mathrm{AMal}(\mathcal{C}_g)\longrightarrow \mathbf{Ab}(\mathcal{C}),
\]
where, as in Section \ref{sec:direction}, $\mathcal{C}_g\subseteq\mathcal{C}$ denotes the full subcategory of objects with global support. The internal abelian groups in $\mathcal{C},$ appearing as the codomain of $d,$ are themselves simply the pointed objects in $\mathrm{Mal}(\mathcal{C}).$ Thus, our first step will be to describe the internal Mal'tsev algebras in $Sl(\mathbf{Cat}_B/\underline{Y}).$
\begin{lemma}
\label{lemma:q_comp}
Consider a commutative triangle
\begin{equation*}
\xymatrixcolsep{1.5pc}
\xymatrix{
{\un{X}} \ar[rr]^-{\un{g}} \ar[rd]_-{\un{f}}  &\ &{\un{Z}} \ar[ld]^-{\un{h}} \\
&{\un{Y}} &\
}
\end{equation*}
in $\mathbf{Cat}_B,$ where $\un{f}$ and $\un{h}$ are $S$-special, with Schreier retractions $q_{\un{f}}$ and $q_{\un{h}}$ respectively. Then, for all $(x,y)\in\mathrm{Eq}(f),$ we have
\[
g\big(q_{\un{f}}(x,y)\big)=q_{\un{h}}\big(g(x),g(y)\big).
\]
\end{lemma}
\begin{proof}
Let $x,y\colon a\rightarrow b$ be such that $(x,y)\in\mathrm{Eq}(f).$ Then $h\Big(g\big(q_{\un{f}}(x,y)\big)\Big)=f\big(q_{\un{f}}(x,y)\big)=1_b$ and $g\big(q_{\un{f}}(x,y)\big)\cdot g(x)=g\big(q_{\un{f}}(x,y)\cdot x\big)=g(y),$ which are the defining properties of $q_{\un{h}}\big(g(x),g(y)\big).$
\end{proof}
\begin{proposition}
\label{prop:char_mal'tsev_objects_slice}
An $S$-special morphism $\un{f}=(f,1_B)\colon\un{X}\rightarrow\un{Y}$ in $\mathbf{Cat}_B$ admits an internal Mal'tsev operation in $Sl(\mathbf{Cat}_B/\underline{Y})$ if and only if, for all $b\in B,$ the group $f^{-1}(1_b)$ is abelian.
\end{proposition}
\begin{proof}
To define a Mal'tsev operation on $\un{f}$ in $Sl(\mathbf{Cat}_B/\underline{Y}),$ one needs first a morphism $p_{\un{f}}=(p_f,1_B)\colon\un{f}\times_Y\un{f}\times_Y\un{f}\rightarrow\un{f},$ where $\un{f}\times_Y\un{f}\times_Y\un{f}$ is an internal functor
\begin{equation}
\begin{aligned}
\label{eqn:fxfxf}
\xymatrix{
{X \times_YX\times_Y X} \ar[d]_-{f\times_Yf\times_Yf} \ar@<.9ex>[r]^-{v_0} \ar@<-.9ex>[r]_-{v_1} &B \ar@{=}[d] \ar[l]|(.28){w_0} \\
Y \ar@<.9ex>[r]^-{y_0} \ar@<-.9ex>[r]_-{y_1} &{B,} \ar[l]|{t_0}
}
\end{aligned}
\end{equation}
and the following diagram commutes:
\begin{equation}
\begin{aligned}
\label{eqn:p_f}
\xymatrix{
{X \times_YX\times_Y X} \ar[rr]^-{p_f} \ar[rd]_-{f\times_Yf\times_Yf \ \ } &\ &X \ar[ld]^-f \\
&{Y.} &\
}
\end{aligned}
\end{equation}
The triple product $f\times_Yf\times_Yf$ can be realized by $X \times_YX\times_Y X=\{(x,y,z)\in X^3:f(x)=f(y)=f(z)\}\rightarrow Y,$ $(x,y,z)\mapsto f(x),$ and $p_f$ is required to satisfy the Mal'tsev equations $p_f(x,y,y)=x$ and $p_f(x,x,y)=y.$ In this setting, we have $v_0\colon(x,y,z)\mapsto x_0(x)=x_0(y)=x_0(z),$ $v_1\colon(x,y,z)\mapsto x_1(x)=x_1(y)=x_1(z),$ and $w_0\colon b\mapsto(1_b,1_b,1_b).$

Suppose that such a Mal'tsev operation $p_{\un{f}}$ exists, and consider arrows $x,y\colon b\rightarrow b$ in $X$ such that $f(x)=f(y)=1_b.$ Then
\begin{equation*}
\begin{split}
x\cdot y&=p_f(x,1_b,1_b)\cdot p_f(1_b,1_b,y)\\
&=p_f\big((x,1_b,1_b)\cdot(1_b,1_b,y)\big)\\
&=p_f(x,1_b,y)\\
&=p_f\big((1_b,1_b,y)\cdot(x,1_b,1_b)\big)\\
&=p_f(1_b,1_b,y)\cdot p_f(x,1_b,1_b)\\
&=y\cdot x,
\end{split}
\end{equation*}
so that $f^{-1}(1_b)$ is abelian.

Conversely, observe that if a morphism $p_{\un{f}}=(p_f,1_B)$ is given, as in the diagram \eqref{eqn:p_f} above, by Lemma \ref{lemma:q_comp} we must have
\begin{equation}
\label{eqn:quasi_q}
p_f\Big(q_{\un{f}\times_Y\un{f}\times_Y\un{f}}\big((x^\prime,y^\prime,z^\prime),(x,y,z)\big)\Big)=q_{\un{f}}\big(p_f(x^\prime,y^\prime,z^\prime),p_f(x,y,z)\big)
\end{equation}
(because $\un{f},$ and thus the product $\un{f}\times_Y\un{f}\times_Y\un{f}$ in $Sl(\mathbf{Cat}_B/\un{Y}),$ are $S$-special).

Let us compute the Schreier retraction $q_{\un{f}\times_Y\un{f}\times_Y\un{f}}.$

For all $\big((x^\prime,y^\prime,z^\prime),(x,y,z)\big)\in\mathrm{Eq}(f\times_Yf\times_Yf),$ the arrow $q_{\un{f}\times_Y\un{f}\times_Y\un{f}}\big((x^\prime,y^\prime,z^\prime),(x,y,z)\big)$ is defined by the property
\[
q_{\un{f}\times_Y\un{f}\times_Y\un{f}}\big((x^\prime,y^\prime,z^\prime),(x,y,z)\big)\cdot (x^\prime,y^\prime,z^\prime)=(x,y,z):
\]
but then, as
\[
(x,y,z)=\big(q_{\un{f}}(x^\prime,x)\cdot x^\prime,q_{\un{f}}(y^\prime,y)\cdot y^\prime,q_{\un{f}}(z^\prime,z)\cdot z^\prime\big),
\]
it follows that
\[
q_{\un{f}\times_Y\un{f}\times_Y\un{f}}\big((x^\prime,y^\prime,z^\prime),(x,y,z)\big)=\big(q_{\un{f}}(x^\prime,x),q_{\un{f}}(y^\prime,y),q_{\un{f}}(z^\prime,z)\big),
\]
and the equation \eqref{eqn:quasi_q} becomes
\begin{equation}
\label{eqn:real_q}
p_f\big(q_{\un{f}}(x^\prime,x),q_{\un{f}}(y^\prime,y),q_{\un{f}}(z^\prime,z)\big)=q_{\un{f}}\big(p_f(x^\prime,y^\prime,z^\prime),p_f(x,y,z)\big).
\end{equation}
Thus, we can recover $p_f(x,y,z)$ as
\begin{equation*}
\begin{split}
p_f(x,y,z)&=q_{\un{f}}\big(p_f(x^\prime,y^\prime,z^\prime),p_f(x,y,z)\big)\cdot p_f(x^\prime,y^\prime,z^\prime)\\
&=p_f\big(q_{\un{f}}(x^\prime,x),q_{\un{f}}(y^\prime,y),q_{\un{f}}(z^\prime,z)\big)\cdot p_f(x^\prime,y^\prime,z^\prime),
\end{split}
\end{equation*}
and setting $z^\prime=z$ and $x^\prime=y=y^\prime$ we get
\begin{equation*}
\begin{split}
p_f(x,y,z)&=p_f\big(q_{\un{f}}(y,x),q_{\un{f}}(y,y),q_{\un{f}}(z,z)\big)\cdot p_f(y,y,z)\\
&=p_f\big(q_{\un{f}}(y,x),1_b,1_b\big)\cdot z\\
&=q_{\un{f}}(y,x)\cdot z
\end{split}
\end{equation*}
(using Proposition \ref{prop:prop_retraction}(2)).

This means that if a Mal'tsev operation $p_{\un{f}}=(p_f,1_B)$ on $\un{f}$ exists, $p_f$ must be defined by
\begin{equation}
\label{eqn:eq_p}
p_f(x,y,z)=q_{\un{f}}(y,x)\cdot z.
\end{equation}
This map surely makes the diagrams \eqref{eqn:fxfxf} and \eqref{eqn:p_f} commute and satisfies the Mal'tsev equations $p_f(x,y,y)=x$ and $p_f(x,x,y)=y$ (using again Proposition \ref{prop:prop_retraction}(2) and the definition of $q_{\un{f}}$), but in order for it to give an internal functor $p_{\un{f}}=(p_f,1_B)$ it must also preserve the compositions. We maintain that this is true precisely when all the groups $f^{-1}(1_b)$ are abelian.

Indeed, in the latter case, we claim that for all $(u,v)\in\mathrm{Eq}(f),$ $u,v\colon a\rightarrow b,$ and all $k\in f^{-1}(1_a),$ the equality
\begin{equation}
\label{eqn:q_if_ker_ab}
q_{\un{f}}(u,u\cdot k)=q_{\un{f}}(v,v\cdot k)
\end{equation}
holds.

Assuming this, it is then easy to prove that for all $x,y,z\colon a\rightarrow b$ and $x^\prime,y^\prime,z^\prime\colon b\rightarrow c$ in $X$ such that $f(x)=f(y)=f(z)$ and $f(x^\prime)=f(y^\prime)=f(z^\prime)$ one has indeed:
\begin{equation*}
\begin{split}
p_f(x^\prime,y^\prime,z^\prime)\cdot p_f(x,y,z)&=q_{\un{f}}(y^\prime,x^\prime)\cdot z^\prime\cdot q_{\un{f}}(y,x)\cdot z \\
&=q_{\un{f}}(y^\prime,x^\prime)\cdot\Big(q_{\un{f}}\big(z^\prime,z^\prime\cdot q_{\un{f}}(y,x)\big)\cdot z^\prime\Big)\cdot z\\
&=q_{\un{f}}(y^\prime,x^\prime)\cdot q_{\un{f}}\big(y^\prime,y^\prime\cdot q_{\un{f}}(y,x)\big)\cdot z^\prime\cdot z \ \ \ \ \txt{(by \eqref{eqn:q_if_ker_ab})}\\
&=q_{\un{f}}\big((y^\prime,x^\prime)\cdot (y,x)\big)\cdot z^\prime\cdot z \ \ \ \ \txt{(by Proposition \ref{prop:prop_retraction}(3))}\\
&=q_{\un{f}}(y^\prime\cdot y,x^\prime\cdot x)\cdot z^\prime\cdot z\\
&=p_f(x^\prime\cdot x,y^\prime\cdot y,z^\prime\cdot z)\\
&=p_f\big((x^\prime,y^\prime,z^\prime)\cdot(x,y,z)\big).
\end{split}
\end{equation*}
To prove \eqref{eqn:q_if_ker_ab}, observe first that
\begin{equation*}
q_{\un{f}}(u,v)\cdot q_{\un{f}}(u,u\cdot k)\cdot u=q_{\un{f}}(u,v)\cdot u\cdot k=v\cdot k
\end{equation*}
and similarly
\begin{equation*}
q_{\un{f}}(v,v\cdot k)\cdot q_{\un{f}}(u,v)\cdot u=q_{\un{f}}(v,v\cdot k)\cdot v=v\cdot k,
\end{equation*}
so that
\[
q_{\un{f}}(u,v)\cdot q_{\un{f}}(u,u\cdot k)=q_{\un{f}}(u,v\cdot k)=q_{\un{f}}(v,v\cdot k)\cdot q_{\un{f}}(u,v).
\]
Being the group $f^{-1}(1_b)$ abelian by assumption, we can cancel $q_{\un{f}}(u,v)$ from the last equation, obtaining \eqref{eqn:q_if_ker_ab}.

\end{proof}
\begin{remark}
It follows that the groupoid $K(\un{f})$ given by the pullback \eqref{eqn:K_f}, where $\un{f}$ is an internal Mal'tsev algebra in $Sl(\mathbf{Cat}_B/\un{Y}),$ is an abelian groupoid in the sense of \cite{aspherical}.
\end{remark}
Since a morphism $\un{f}=(f,1_B)\colon\un{X}\rightarrow\un{Y}$ in $\mathbf{Cat}_B$ is a regular epimorphism if and only if the function $f$ is surjective, and since the identity functor $1_{\un{Y}}$ is a terminal object in $Sl(\mathbf{Cat}_B/\un{Y}),$ the object $\un{f}\in Sl(\mathbf{Cat}_B/\un{Y})$ has global support if and only if the map $f$ is surjective.

Thus, the domain of the direction functor $d$ of $Sl(\mathbf{Cat}_B/\un{Y})$ is given by all $S$-special morphisms $\un{f}\colon\un{X}\rightarrow\un{Y}$ with $f$ surjective, and such that all the fibres $f^{-1}(1_b)$ are abelian groups.

As for the codomain of $d,$ we know that the internal abelian groups in $Sl(\mathbf{Cat}_B/\un{Y})$ are precisely those internal Mal'tsev algebras $\xymatrixcolsep{1pc}\xymatrix{{{\un{g}}\colon\un{Z}} \ar[r] &{\un{Y}}}$ admitting a morphism $\xymatrixcolsep{1pc}\xymatrix{{\un{t}\colon1_{\un{Y}}} \ar[r] &{\un{g}}}$ in  $Sl(\mathbf{Cat}_B/\un{Y}),$ which implies that $\xymatrix{{\un{Z}} \ar@<.5ex>[r]^-{\un{g}} &{\un{Y}} \ar@<.5ex>[l]^-{\un{t}}}$ is a point in $\mathbf{Cat}_B$: since $\un{g}$ is an $S$-special morphism, $(\un{g},\un{t})$ is necessarily a Schreier point (cf. Remark \ref{rmks:s-special}(4)), which by Proposition \ref{prop:schreier_functors} corresponds to a functor $\un{Y}\rightarrow \mathbf{Ab}$ (as in \eqref{eqn:functor_sch_point}). With the terminology of \cite{cocat}, we conclude that the codomain of $d$ is the category $\mathrm{Mod}_{\un{Y}}$ of $\un{Y}$-modules (whose morphisms are the natural transformations).

Now, observe that if $\un{f}\colon\un{X}\rightarrow\un{Y}$ is an object in $\mathrm{Mal}\big(Sl(\mathbf{Cat}_B/\un{Y})\big)_g,$ the equality \eqref{eqn:q_if_ker_ab} entails that for all $y\colon a\rightarrow b$ in $Y$ and all $k\in f^{-1}(1_a)$ the arrow ${^yk}=q_{\un{f}}(x,x\cdot k)\in f^{-1}(1_b),$ where $x\in X$ is any element such that $f(x)=y,$ does not depend on the choice of $x$: thus the functor \eqref{eqn:functor_s-special} associated with the $S$-special morphism $\un{f}$ induces a $\un{Y}$-module
\begin{equation}
\label{eqn:d(f)}
D_{\un{f}}\colon\underline{Y} \rightarrow \mathbf{Ab}, \ \  (y\colon a\rightarrow b) \mapsto \big({^y()}\colon f^{-1}(1_a)\rightarrow f^{-1}(1_b), \ k\mapsto {^yk}\big).
\end{equation}
We claim that $D_{\un{f}}$ is precisely the direction of the internal Mal'tsev algebra $\un{f},$ i.e. that one has $d(\un{f})=D_{\un{f}}.$

Recall from Section \ref{sec:direction} that if $\mathcal{C}$ is a Barr-exact and Mal'tsev category, and $X\in\mathrm{Mal}(\mathcal{C}_g),$ then the direction $d(X)\in\mathbf{Ab}(\mathcal{C})$ is determined as the only internal (necessarily abelian) group in $\mathcal{C}$ which admits a simply transitive group action on $X$: thus, to prove that $d(\un{f})=D_{\un{f}}$ when $\mathcal{C}=Sl(\mathbf{Cat}_B/\un{Y}),$ it suffices to produce a simply transitive action of the group $D_{\un{f}}$ on $\un{f}.$

Under the equivalence between Schreier points in $\mathbf{Cat}_B$ and $\mathbf{Mon}$-valued functors (Proposition \ref{prop:schreier_functors}), the $\un{Y}$-module $D_{\un{f}}$ corresponds to the point
\begin{equation*}
\begin{aligned}
{
\xymatrix{
{\underline{Z}\cong K(\un{f})\rtimes\un{Y}} \ar@<-.5ex>[d]_-{\underline{\delta}_{\un{f}}}\\
{\underline{Y}} \ar@<-.5ex>[u]_-{\underline{\upsilon}_{\un{f}}}
}
}
{\xymatrix{: \\ :}}
{
\xymatrix{
{Z} \ar@<-.5ex>[d]_-{\delta_{\un{f}}} \ar@<.9ex>[r]^-{z_0} \ar@<-.9ex>[r]_-{z_1} &{B} \ar@{=}[d] \ar[l]|{w_0} \\
{Y} \ar@<-.5ex>[u]_-{\upsilon_{\un{f}}} \ar@<.9ex>[r]^-{y_0} \ar@<-.9ex>[r]_-{y_1} &{B,} \ar[l]|{t_0}
}
}
\end{aligned}
\end{equation*}
where, explicitly, $Z=\{(k,y):y\in Y \ \text{and} \ k\in f^{-1}(1_{y_1(y)})\},$ $\delta_{\un{f}}(k,y)=y$ and $\upsilon_{\un{f}}(y)=(1_{y_1(y)},y)$ (plus, of course, $z_0(k,y)=y_0(y),$ $z_1(k,y)=y_1(y),$ and $w_0(b)=(1_b,1_b)$).

It is easy to verify that the morphism $\un{\delta}_{\un{f}}$ is $S$-special, and that it is indeed an internal group in $Sl({\mathbf{Cat}_B})/\un{Y}$ with group operation $(k,k^\prime,y)\mapsto (k^\prime\cdot k,y)$ (for all $y\colon a\rightarrow b$ in $Y$ and $k,k^\prime\in f^{-1}(1_b)$).

To obtain a simply transitive action of the group $D_{\un{f}}$ on $\un{f}$ in $Sl(\mathbf{Cat}_B/\un{Y}),$ we need first of all an isomorphism
\begin{equation}
\label{eqn:phi_f}
\xymatrix{ {\un{\phi}_{\un{f}}: \un{f}\times_{\un{Y}} \un{f} } \ar[r]^-{\sim} &{ \un{f}\times_{\un{Y}} \un{\delta}_{\un{f}} } }
\end{equation}
in $Sl(\mathbf{Cat}_B/\un{Y})$ such that the triangle
\begin{equation}
\begin{aligned}
\label{eqn:phi_f_slice}
\xymatrixcolsep{1.5pc}
\xymatrix{
{ \un{f}\times_{\un{Y}} \un{f} } \ar[rr]^-{\un{\phi}_{\un{f}}}_-{\sim} \ar[rd]_-{\un{\pi}_1} &\ &{ \un{f}\times_{\un{Y}} \un{\delta}_{\un{f}} } \ar[ld]^-{\un{\pi}_1} \\
&{\un{f}} &\
}
\end{aligned}
\end{equation}
commutes (where the $\un{\pi}_1$'s denote the first projections).

The product $\un{f}\times_{\un{Y}} \un{f}$ in $Sl(\mathbf{Cat}_B/\un{Y})$ is completely determined by the kernel pair $\mathrm{Eq}(f),$ and similarly the product $\un{f}\times_{\un{Y}} \un{\delta}_{\un{f}}$ is completely determined by the pullback $X\times_YZ\cong\{(x,k,y):x\in X, \  k\in f^{-1}(1_{x_1(x)}), \ \text{and} \ y=f(x)\}$ of $\delta_{\un{f}}$ along $f$ in $\mathbf{Set}.$

Define then
\begin{equation}
\label{eqn:simply_trans_action}
\phi_{\un{f}}:\mathrm{Eq}(f) \longrightarrow X\times_Y Z, \ (x,v)\mapsto (x,q_{\un{f}}(x,v),f(x)).
\end{equation}
The map $\phi_{\un{f}}$ is injective, because if $q_{\un{f}}(x,v)=q_{\un{f}}(x,u),$ then $v=q_{\un{f}}(x,v)\cdot x=q_{\un{f}}(x,u)\cdot x=u,$ and it is surjective, because for all $x\colon a\rightarrow b$ in $X$ and all $k\in f^{-1}(1_b)$ one has $k=q_{\un{f}}(x,k\cdot x).$

To prove that $\phi_{\un{f}}$ preserves compositions, compute for all $x,v\colon a\rightarrow b,$ $x^\prime,v^\prime\colon b\rightarrow c$ such that $f(x)=f(v)$ and $f(x^\prime)=f(v^\prime)$:
\begin{equation*}
\begin{split}
\phi_{\un{f}}(x^\prime,v^\prime)\cdot \phi_{\un{f}}(x,v)&=\big(x^\prime,q_{\un{f}}(x^\prime,v^\prime),f(x^\prime)\big)\cdot\big(x,q_{\un{f}}(x,v),f(x)\big) \\
&=\Big(x^\prime\cdot x, q_{\un{f}}(x^\prime,v^\prime)\cdot D_{\un{f}}\big(f(x^\prime)\big)\big(q_{\un{f}}(x,v)\big),f(x^\prime)\cdot f(x) \Big) \\
&=\Big(x^\prime\cdot x, q_{\un{f}}(x^\prime,v^\prime)\cdot q_{\un{f}}(x^\prime,x^\prime\cdot q_{\un{f}}(x,v)\big),f(x^\prime\cdot x)\Big)\\
&=\big(x^\prime\cdot x, q_{\un{f}}(x^\prime\cdot x,v^\prime\cdot v), f(x^\prime\cdot x) \big) \ \ \ \ \txt{(by Proposition \ref{prop:prop_retraction}(3))}\\
&=\phi_{\un{f}}(x^\prime\cdot x,v^\prime\cdot v)\\
&=\phi_{\un{f}}\big((x^\prime,v^\prime)\cdot (x,v)\big).
\end{split}
\end{equation*}
The bijection $\phi_{\un{f}}$ is easily seen to result in an isomorphism $\un{\phi}_{\un{f}}$ of small categories, which realizes the required isomorphism \eqref{eqn:phi_f} in $Sl(\mathbf{Cat}_B/\un{Y})$ (the commutativity of the triangle \eqref{eqn:phi_f_slice} is immediate by construction). Moreover, $\un{\phi}_{\un{f}}^{-1}=(\phi_{\un{f}}^{-1},1_B),$ where
\[
\phi_{\un{f}}^{-1}\colon f\times_Y \delta_{\un{f}}=X\times_Y Z\rightarrow \mathrm{Eq}(f)=f\times_Y f, \ (x,k,f(x))\mapsto (x,k\cdot x),
\]
is clearly a (right) group action with respect to the group structure on $\un{\delta}_{\un{f}}.$

 This proves that the direction functor of $Sl(\mathbf{Cat}_B/\un{Y})$ is given by
\begin{equation}
\label{eqn:direction_functor_eq}
d\colon\mathrm{Mal}\big(Sl(\mathbf{Cat}_B/\un{Y})\big)_g\longrightarrow \mathrm{Mod}_{\un{Y}}, \ \un{f}\mapsto d(\un{f})=D_{\un{f}} \ \text{as in } \eqref{eqn:d(f)}.
\end{equation}
\begin{remark}
\label{rmk:d_on_morphisms}
For any two objects $\un{f},\un{g}\in\mathrm{Mal}\big(Sl(\mathbf{Cat}_B/\un{Y})\big)_g,$ denote by $\un{\delta}_{\un{f}}\colon\un{Z}\rightarrow\un{Y}$ and $\un{\delta}_{\un{g}}\colon\un{M}\rightarrow\un{Y}$ the internal abelian groups in $Sl(\mathbf{Cat}_B/\un{Y})$ corresponding to the $\un{Y}$-modules $d(\un{f})=D_{\un{f}}$ and $d(\un{g})=D_{\un{g}}$ respectively, as above. By definition, the arrows $a\rightarrow b$ in $\un{Z}\cong K(\un{f})\rtimes\un{Y}$ are the pairs $(k,y)$ with $y\colon a\rightarrow b$ in $Y$ and $k\in f^{-1}(1_b),$ and similarly the arrows $a\rightarrow b$ in $\un{M}\cong K(g)\rtimes\un{Y}$ are the pairs $(k^\prime,y^\prime)$ with $y^\prime\colon a\rightarrow b$ in $Y$ and $k^\prime\in g^{-1}(1_b).$ Then, if $\un{\vartheta}=(\vartheta,1_B)\colon\un{f}\rightarrow\un{g}$ is a morphism in $\mathrm{Mal}\big(Sl(\mathbf{Cat}_B/\un{Y})\big)_g,$ it is not difficult to see that the morphism $d(\un{\vartheta})=(d(\vartheta),1_B)\colon\un{Z}\rightarrow\un{M}$ is given by $d(\vartheta)(k,y)=(\vartheta(k),y),$ and it is thus completely determined by the restriction of $\vartheta$ to $K(f)=\bigsqcup_{b\in B}f^{-1}(1_b).$
\end{remark}

\subsection{Extensions of small categories and the second cohomology group}
In \cite{cocat, hoff}, an extension of a small category $\un{C}$ by a small category $\un{K}$ is defined as a sequence
\begin{equation}
\label{eqn:ext_of_cat}
\xymatrix{ {\un{K}} \ar[r]^-{\un{i}} &{\un{H}} \ar[r]^-{\un{p}} &{\un{C}}}
\end{equation}
where $\un{H}$ is a small category, $\un{i}$ is a faithful functor (so that $\un{K}$ can be realized as a subcategory of $\un{H}$), and $\un{p}$ is a bijective-on-objects full functor, satisfying the following condition:
\begin{equation}
\tag{$\ast$}
\begin{split}
 \text{for all } &\text{ arrows } h \text{ and } h^\prime \text{ in } \un{H}, \text{ the equality } \un{p}(h)=\un{p}(h^\prime) \text{ holds if and only if } \\ &\text{ there exists a unique arrow } k \text{ of } \un{K} \text{ such that } h^\prime=i(k)\cdot h.
\end{split}
\end{equation}

Clearly, there is no harm in assuming $\un{p}$ to be the identity on objects, so that $\un{p}=(p,1_B)$ is a morphism in $\mathbf{Cat}_B$  for some $B=ob(\un{H})=ob(\un{C}),$ and by $(\ast)$ one has at once:
\begin{itemize}
\item $ob(\un{K})=ob(\un{H})=B,$ and we can identify $i(k)$ with $k$ for every arrow $k$ of $\un{K},$ where $\un{i}=(i,1_B);$
\item if $a,b\in B$ and $a\neq b,$ there is no morphism $k\colon a\rightarrow b$ in $\un{K},$ so that $\un{K}$ is totally disconnected;
\item for all $k\colon b\rightarrow b$ in $\un{K},$ one has $p(k)=1_b;$
\item for all $b\in B,$ the set $\un{K}(b,b)$ of endomorphisms of $b$ in $\un{K}$ is a group;
\item if $h\in \un{H}(b,b)$ and $p(h)=1_b,$ then $h\in\un{K}(b,b).$
\end{itemize}
Thus, in the language of the previous section, an extension of categories \eqref{eqn:ext_of_cat} is simply an $S$-special regular epimorphism $\un{p}$ in $\mathbf{Cat}_B,$ together with $\un{K}=K(\un{p})$ as in \eqref{eqn:K_f}.

Always in the terminology of \cite{cocat} (see also \cite{golasinski}), if $\un{C}\in\mathbf{Cat}_B$ and $A\colon\un{C}\rightarrow\mathbf{Ab}$ is a $\un{C}$-module, an extension of $\un{C}$ by $A$ is an extension of categories \eqref{eqn:ext_of_cat} such that the set of arrows of $\un{K}$ is given by $A^+=\bigsqcup_{b\in B}A(b).$

Since by definition $D_{\un{f}}^+=K(\un{f})$ for all $\un{f}\in\mathrm{Mal}\big(Sl(\mathbf{Cat}_B/\un{Y})\big)_g,$ we conclude that the set of extensions of a small category $\un{Y}\in\mathbf{Cat}_B$ by a $\un{Y}$-module $A$ in the sense of \cite{cocat, golasinski} coincides with the fibre $d^{-1}(A)$ of the direction functor \eqref{eqn:direction_functor_eq}.

Two extensions $\xymatrix{{\mathbb{E}:K(\un{f})\cong\un{A}^+} \ar@{>->}[r] &{\un{X}} \ar[r]^-{\un{f}} &{\un{Y}}}$ and $\xymatrix{{\mathbb{E}^\prime:K(\un{f}^\prime)\cong\un{A}^+} \ar@{>->}[r] &{\un{X}^\prime} \ar[r]^-{\un{f}^\prime} &{\un{Y}}}$ of $\un{Y}$ by $A$ are called equivalent in \cite{cocat, golasinski, hoff} if there exists a morphism $\un{\varphi}$ making the diagram
\begin{equation*}
\xymatrix{
{\mathbb{E}:\un{A}^+} \ar@<.9ex>@{=}[d] \ar@{>->}[r] &{\un{X}} \ar[d]^-{\un{\varphi}} \ar[r]^-{\un{f}} &{\un{Y}} \ar@{=}[d] \\
{\mathbb{E}^\prime:\un{A}^+} \ar@{>->}[r] &{\un{X}^\prime} \ar[r]_-{\un{f}^\prime} &{\un{Y}}
}
\end{equation*}
commute. Then $\un{\varphi}$ is necessarily an isomorphism, thanks to the Short Five Lemma for $S$-special morphisms (whose proof in $\mathbf{Cat}_B$ is just the same as the one given for monoids in \cite{schreier_book}, Proposition 7.2.1, or \cite{MMS18}, Proposition 2.12). Using a notation akin to that of \cite{homology} for the cohomology of groups, the set of equivalence classes of extensions of $\un{Y}$ by $A$ is denoted by $\mathrm{Opext}^1(\un{Y},A)$ in \cite{golasinski}.

Since the direction functor $d$ is conservative, the existence of $\un{\varphi}$ as above is equivalent to $\un{f}$ and $\un{f}^\prime$ being in the same connected component of the fibre $d^{-1}(A),$ in the sense of \eqref{eqn:zig-zag}: thus $\mathrm{Opext}^1(\un{Y},A)=\pi_0\big(d^{-1}(A)\big)$ as sets. To see that they also coincide as groups, recall that the group operation in $\mathrm{Opext}^1(\un{Y},A)$ is given by the Baer sum of (equivalence classes of) extensions, which is defined as follows. Starting with $\mathbb{E}$ and $\mathbb{E}^\prime$ as above, consider first the pullback $\un{P}=(\un{X}\times\un{X}^\prime)\times_{\un{Y}\times\un{Y}}\un{Y}$ in $\mathbf{Cat}_B$ of the product $\un{f}\times\un{f}^\prime$ along the diagonal $\Delta_{\un{Y}}$:
\[
\xymatrix{
{\un{P}} \ar[d] \ar[r]^-{\un{\psi}} \pullback &{\un{Y}} \ar[d]^-{\Delta_{\un{Y}}} \\
{\un{X}\times\un{X}^\prime} \ar[r]_-{\un{f}\times\un{f}^\prime} &{\un{Y}\times\un{Y}.}
}
\]
At the level of arrows, $\un{P}$ can be realized as the pullback $P=\{(x,x^\prime)\in X\times X^\prime:f(x)=f^\prime(x^\prime)\}$ of $f$ along $f^\prime,$ with $\psi\colon(x,x^\prime)\mapsto f(x)=f^\prime(x^\prime).$ We end up with an extension
\begin{equation*}
\xymatrix{{\mathbb{E}\times_{\un{Y}}\mathbb{E}^\prime:K(\un{\psi})} \ar@{>->}[r]^-{\un{\iota}} &{\un{P}} \ar[r]^-{\un{\psi}} &{\un{Y}}}
\end{equation*}
corresponding to the product $\un{f}\times_{\un{Y}}\un{f}^\prime$ in $\mathrm{Mal}\big(Sl(\mathbf{Cat}_B/\un{Y})\big),$ where $K(\psi)\cong\bigsqcup_{b\in B}\big(A(b)\times A(b)\big).$  In particular, we have $d(\un{f}\times_{\un{Y}}\un{f}^\prime)\cong A\times A.$

Next comes the pushforward of $\mathbb{E}\times_{\un{Y}}\mathbb{E}^\prime$ along the codiagonal $\nabla_{\un{A}^+},$ given on the arrows by $\nabla_{A^+}\colon\bigsqcup_{b\in B}\big(A(b)\times A(b)\big)\longrightarrow\bigsqcup_{b\in B}A(b)=A^+,$ $(k,k^\prime)\mapsto k^\prime\cdot k$ (the product in the group $A(b)$). The pushforward construction, which in the abelian case coincides with the pushout and which is sometimes also referred to as a pushout on account of this fact (see for example \cite{holt, golasinski}), is performed in the following way. Observe that the $\un{Y}$-module $A\colon\un{Y}\rightarrow\mathbf{Ab}$ induces a $\un{P}$-module $A_{\un{P}}\colon\un{P}\rightarrow\mathbf{Ab}$ by sending a morphism $(x,x^\prime)\colon a\rightarrow b$ to $A\big(f(x)=f^\prime(x^\prime)\big)\colon A(a)\rightarrow A(b).$ By Proposition \ref{prop:schreier_functors}, the functor $A_{\un{P}}$ corresponds to a Schreier point on $\un{P}$ in $\mathbf{Cat}_B,$ whose domain has $W=\{(\xi,x,x^\prime):(x,x^\prime)\in P, \ (x,x^\prime)\colon a\rightarrow b, \ \text{and } \xi\in A(b)\}$ as the set of arrows, and we obtain a square
\begin{equation*}
\xymatrix{
{K(\un{\psi})} \ar[d]_-{\nabla_{\un{A}^+}} \ar@{>->}[r]^-{\un{\iota}} &{\un{P}} \ar[d]^-{\un{\sigma}} \\
{\un{A}^+} \ar[r]_-{\un{j}} &{\un{W}\cong \un{A}^+\rtimes \un{P},}
}
\end{equation*}
where $j\big(\xi\in A(b)\big)=(\xi,1_b,1_b)$ and $\sigma\big((x,x^\prime):a\rightarrow b\big)=(e_{A(b)},x,x^\prime).$ This square is not commutative, in general.

Consider a coequalizer $\un{r}\colon\un{W}\longrightarrow\un{Q}$ of $\un{\sigma}\un{\iota}$ and $\un{j}\nabla_{\un{A}^+},$ which can be described (at the level of arrows) by the quotient set $Q=W/_{\sim}$ of $W$ under the equivalence relation generated by imposing $(k^\prime\cdot k,1_b,1_b)\sim(e_{A(b)},k,k^\prime).$

Then one proves that the map $\eta\colon Q\longrightarrow Y,$ $[(k,x,x^\prime)]_{\sim}\mapsto f(x)=f^\prime(x^\prime),$ is well defined, and eventually we get an extension
\begin{equation}
\label{eqn:baer_sum}
\xymatrix{{\mathbb{E}+_{\un{Y}}\mathbb{E}^\prime:\un{A}^+} \ar@{>->}[r]^-{\un{r}\un{j}} &{\un{Q}} \ar[r]^-{\un{\eta}} &{\un{Y}}}
\end{equation}
of $\un{Y}$ by $A,$ which is the Baer sum of $\mathbb{E}$ and $\mathbb{E}^\prime,$ together with a morphism of extensions
\begin{equation}
\begin{aligned}
\label{eqn:cocart_on_product}
\xymatrixcolsep{2.5pc}
\xymatrix{
{\mathbb{E}\times_{\un{Y}}\mathbb{E}^\prime:K(\un{\psi})} \ar@<2.4pc>[d]_-{\nabla_{\un{A}^+}} \ar@{>->}[r]^-{\un{\iota}} &{\un{P}} \ar[d]^-{\un{\mu}=\un{r}\un{\sigma}} \ar[r]^-{\un{\psi}} &{\un{Y}} \ar@{=}[d] \\
{\mathbb{E}+_{\un{Y}}\mathbb{E}^\prime: \ \ \ \un{A}^+} \ar@{>->}[r]_-{\un{r}\un{j}} &{\un{Q}} \ar[r]_-{\un{\eta}} &{\un{Y}.}
}
\end{aligned}
\end{equation}
(A detailed account on the above construction in the case $B=1$ can be found in \cite{MMS18}.)

By Remark \ref{rmk:d_on_morphisms}, the morphism $\un{\mu}$ in \eqref{eqn:cocart_on_product} is sent by the direction functor to the group operation on (the internal abelian group $\un{\delta}_{\un{f}}\cong\un{\delta}_{\un{f}^\prime}$ in $Sl(\mathbf{Cat}_B/\un{Y})$ corresponding to) the $\un{Y}$-module $A,$ and since $d=d_{\un{Y}}$ is a conservative cofibration we conclude that indeed the Baer sum \eqref{eqn:baer_sum} coincides with the tensor product $\un{f}\otimes\un{f}^\prime$ induced in the fibre $d_{\un{Y}}^{-1}(A)$ as in Proposition \ref{prop:fibres_abelian_groups}.

By \cite{golasinski}, Theorem 1.2, the group $\mathrm{Opext}^1(\un{Y},A)=\pi_0\big(d_{\un{Y}}^{-1}(A)\big)$ is a realization of the second cohomology group $H^2(\un{Y},A)$ in the sense of \cite{cocat}.
\section{The higher cohomology groups} \label{section higher cohomology}
We now turn our attention to the cohomology groups $H^n(\un{Y},A)$ of \cite{cocat, golasinski}, with $n\geq 3,$ and show that they also may be framed in terms of (higher) direction functors $d_n$ on the category $Sl(\mathbf{Cat}_B/\un{Y}).$
\subsection{Crossed modules in $\mathbf{Cat}_B$ and the cohomology groups $H^n(\un{Y},A)$ for $n\geq 3$}
First, recall that a crossed module \cite{whitehead_1, whitehead_2} (in the category $\mathbf{Gp}$ of groups) is a pair $(\delta\colon A\rightarrow X,\omega)$ where $A,$ $X$ are groups, $\delta$ is a group homomorphism and $\omega\colon X\rightarrow\mathrm{Aut}(A)$ is a group action, satisfying the axioms $\delta\big(\omega(x)(a)\big)=x\cdot\delta(a)\cdot x^{-1}$ and $\omega\big(\delta(a)\big)(a')=a\cdot a'\cdot a^{-1}$ for all $a,a'\in A$ and $x\in X.$ The relation between crossed modules and the third cohomology group $H^3(G,M,\omega)$ of a group $G$ with coefficients in the $G$-module $(M,\omega),$ as introduced in \cite{EML}, is studied in \cite{holt, huebschmann}. By a classical result  \cite{BS}, the category $\mathbf{Xmod}$ of crossed modules is equivalent to the category $\mathrm{Cat}(\mathbf{Gp})$ of internal categories in $\mathbf{Gp}.$

If $A$ and $X$ are monoids and $\omega\colon X\rightarrow\mathrm{End}(A)$ is a monoid action of $X$ on $A,$ a monoid homomorphism $\delta$ as above is called a \emph{crossed semimodule} if the analogous, monoid friendly axioms $\delta\big(\omega(x)(a)\big)+x=x+\delta(a)$ and $\omega\big(\delta(a)\big)(a')+a=a+a'$ hold (using an additive notation); this notion was introduced in \cite{JS}, and later independently in \cite{crossed-semimodules}, where it was also shown that the equivalence between crossed modules and internal categories in $\mathbf{Gp}$ extends to an equivalence $\mathbf{Xsmod}\cong\mathrm{SCat}(\mathbf{Mon})$ between crossed semimodules and Schreier internal categories in $\mathbf{Mon}$ (i.e., internal categories $\xymatrix{{\underline{X}:X_2} \ar[r]^-{m_{\un{X}}} &{X_1} \ar@<.9ex>[r]^-{x_0} \ar@<-.9ex>[r]_-{x_1} &{X_0} \ar[l]|-{s_0}}$ such that the pair $(x_0,s_0)$ is a Schreier point).

Armed with the tools of Definition \ref{def:action_Cat_B}, the notion of crossed semimodule can be widened to the general case of $\mathbf{Cat}_B$ in the obvious way:
\begin{definition}[\cite{a1}]
\label{def:mor_of_xsmod}
A crossed semimodule in $\mathbf{Cat}_B$ is a pair $(\un{\delta}\colon\un{A}\rightarrow\un{X},\un{\omega})$ where $\un{\delta}$ is a morphism in $\mathbf{Cat}_B$ and $\un{\omega}\colon\un{X}\rightarrow E(\un{A})$ is an action, satisfying the following axioms:
\begin{enumerate}
\item $\delta\big(\omega(x)(z)\big)\cdot x=x\cdot \delta(z)$ (composition in $\un{X}$) for all $x\in\un{X}(a,b)$ and $z\in\un{A}(a,a),$ for all $a,b\in B;$
\item $\omega\big(\delta(y)\big)(z)\cdot y=y\cdot z$ (composition in $\un{A}$) for all $y\in\un{A}(a,b)$ and $z\in\un{A}(a,a),$ for all $a,b\in B.$
\end{enumerate}
A morphism $(\un{\delta}\colon\un{A}\rightarrow\un{X},\un{\omega})\rightarrow(\un{\delta}'\colon\un{A}'\rightarrow\un{X}',\un{\omega}')$ of crossed semimodules in $\mathbf{Cat}_B$ is then a pair $(\un{\varphi},\un{\psi})$ of morphisms in $\mathbf{Cat}_B$ such that the square
\begin{equation}
\label{eqn:mor_of_xsmod}
\begin{aligned}
\xymatrix{
{\un{A}} \ar[d]_-{\un{\varphi}} \ar[r]^-{\un{\delta}} &{\un{X}} \ar[d]^-{\un{\psi}} \\
{\un{A}'} \ar[r]_-{\un{\delta}'} &{\un{X}'}
}
\end{aligned}
\end{equation}
commutes, and for every $x\in\un{X}(a,b)$ and $z\in\un{A}(a,a)$ (for every $a,b\in B$) the equality
\begin{equation*}
\varphi\big(\omega(x)(z)\big)=\omega'\big(\psi(x)\big)\big(\varphi(z)\big),
\end{equation*}
giving the equivariance with respect to the action, holds.
\end{definition}
We shall denote by $\mathbf{Xsmod}(\mathbf{Cat}_B)$ the category of crossed semimodules and morphisms \eqref{eqn:mor_of_xsmod} in $\mathbf{Cat}_B.$ It is clear that, for $B=1,$ this notion coincides with the one introduced in \cite{JS} for monoids.

As was the case for actions and Schreier points (cf.~Remark~\ref{rmk:actions_and_schreier_points_Cat_B}), it is proven in \cite{a1} that the equivalence between crossed semimodules and Schreier internal categories, valid in $\mathbf{Cat}_1\cong\mathbf{Mon},$ is an instance of a more general adjunction $\Lambda\dashv\Sigma$ between the functors
\begin{equation*}
\Lambda\colon\mathrm{SCat}(\mathbf{Cat}_B)\longrightarrow\mathbf{Xsmod}(\mathbf{Cat}_B), \ (\xymatrix{{\un{Z}} \ar@<.9ex>[r]^-{\un{u}} \ar@<-.9ex>[r]_-{\un{v}} &{\un{X}} \ar[l]|-{\un{w}}})\mapsto \mathcal{K}(\xymatrix{{\un{Z}} \ar@<.5ex>[r]^-{\un{u}} &{\un{X}} \ar@<.5ex>[l]^-{\un{w}}}),
\end{equation*}
where $\mathcal{K}$ is as in \eqref{eqn:K}, and
\begin{equation}
\label{eqn:Sigma}
\Sigma\colon\mathbf{Xsmod}(\mathbf{Cat}_B)\longrightarrow\mathrm{SCat}(\mathbf{Cat}_B), \ (\un{\delta}\colon\un{A}\rightarrow\un{X},\un{\omega})\mapsto (\xymatrix{{\un{A}\rtimes_{\un{\omega}}\un{X}} \ar@<.9ex>[r]^-{\un{\pi}} \ar@<-.9ex>[r]_-{\un{\gamma}} &{\un{X}} \ar[l]|-{\un{\sigma}}}),
\end{equation}
where $\gamma\colon P(\un{A},\un{X})\rightarrow X,$ $(b\xrightarrow{x} b,a\xrightarrow{y} b)\mapsto\delta(x)\cdot y.$ This adjunction is an equivalence if and only if $|B|\leq1.$

Again following \cite{a1}, define a crossed semimodule $(\un{\delta}\colon\un{A}\rightarrow\un{X},\un{\omega})$ in $\mathbf{Cat}_B$ to be a \emph{crossed module} if $\un{A}$ is a totally disconnected groupoid on $B,$ and denote by $\mathbf{Xmod}(\mathbf{Cat}_B)\subseteq\mathbf{Xsmod}(\mathbf{Cat}_B)$ the full subcategory of crossed modules. Then one can prove:
\begin{proposition}[\cite{a1}]
\label{prop:scat_xmod}
The adjunction $\Lambda\dashv\Sigma$ yields an equivalence
\begin{equation*}
\mathrm{SGpd}(\mathbf{Cat}_B)\cong\mathbf{Xmod}(\mathbf{Cat}_B)
\end{equation*}
between Schreier internal groupoids and crossed modules in $\mathbf{Cat}_B.$
\end{proposition}
This equivalence extends to $\mathbf{Cat}_B,$ for an arbitrary set $B,$ the analogous result of \cite{crossed-semimodules} for the case $B=1.$

Our goal, now, is to compute the $1$-dimensional direction functor $d_1\colon 1\text{-}\mathrm{Asp}(\mathcal{C})\rightarrow\mathrm{Ab}(\mathcal{C})$ for the naturally Mal'tsev category $\mathcal{C}=\mathrm{Mal}\big(Sl(\mathbf{Cat}_B/\un{Y})\big),$ and to show that its fibres describe the third cohomology groups $H^3(\un{Y},A)$ of \cite{golasinski} in the same way as the fibres of \eqref{eqn:direction_functor_eq} provide a description for the second cohomology groups $H^2(\un{Y},A).$

To this end, observe that a groupoid in $\mathrm{Mal}\big(Sl(\mathbf{Cat}_B/\un{Y})\big)$ amounts to a commutative diagram
\begin{equation}
\label{eqn:1-grpd-special}
\begin{aligned}
\xymatrix{
{\un{X}_1} \ar[rd]_-{\un{f}_1} \ar@<.9ex>[rr]^-{\un{u}} \ar@<-.9ex>[rr]_-{\un{v}} &\ &{\un{X}_0} \ar[ll]|-{\un{w}} \ar[ld]^-{\un{f}_0} \\
&{\un{Y}}
}
\end{aligned}
\end{equation}
such that $\un{f}_0,$ $\un{f}_1$ are $S$-special morphisms in $\mathbf{Cat}_B,$ the groups $f_1^{-1}(1_b)$ and $f_0^{-1}(1_b)$ are abelian for every $b\in B,$ and $\xymatrix{ {\un{X}_1} \ar@<.9ex>[r]^-{\un{u}} \ar@<-.9ex>[r]_-{\un{v}} &{\un{X}_0} \ar[l]|-{\un{w}}}$ is a Schreier groupoid in $\mathbf{Cat}_B$ (indeed, by Remarks \ref{rmks:s-special} $(3)$ and $(4),$ $\un{u}$ is $S$-special because so are $\un{f}_0$ and $\un{f}_1$ and the equality $\un{f}_0\un{u}=\un{f}_1$ holds, and it follows that $(\un{u},\un{w})$ is a Schreier point). By Proposition \ref{prop:scat_xmod}, the latter corresponds to the crossed module $\xymatrix{ {K(\un{u})} \ar[r]^-{\un{\delta}} &{\un{X}_0} }$ given by the restriction $\un{v}_|$ of $\un{v}$ to $K(\un{u}).$

Next, we know that $\un{f}_0=(f_0,1_B)$ has global support, as an object in $Sl(\mathbf{Cat}_B/\un{Y}),$ if and only if the map $f_0$ is surjective, and moreover we have:
\begin{lemma}
The groupoid \eqref{eqn:1-grpd-special} is connected if and only if, in the crossed module $\un{\delta}=\un{v}_|,$ the image $Im(\delta)$ of the map $\delta\colon \bigsqcup_{b\in B}u^{-1}(1_b)\rightarrow X_0$ coincides with $K(f_0)=\bigsqcup_{b\in B}f_0^{-1}(1_b).$
\end{lemma}
\begin{proof}
By definition, \eqref{eqn:1-grpd-special} is connected if and only if $\la u,v\ra\colon X_1\rightarrow\mathrm{Eq}(f_0)$ is surjective. When this holds and an arrow $x_0\colon b\rightarrow b$ in $X_0$ is such that $f_0(x_0)=1_b,$ there exists $x_1\in u^{-1}(1_b)$ such that $(1_b,x_0)=\la u,v\ra(x_1),$ so that $x_0=v(x_1)=\delta(x_1).$ The fact that $Im(\delta)\subseteq K(f_0)$ is always true, by $f_0u=f_0v.$

Conversely, if $Im(\delta)=K(f_0)$ and $(x_0\colon a\rightarrow b,y_0\colon a\rightarrow b)\in\mathrm{Eq}(f_0),$ we can write $y_0=z_0\cdot x_0$ for some (unique) $z_0\in f_0^{-1}(1_b),$ since $\un{f}_0$ is $S$-special: then, by the assumption, $z_0=\delta(z_1)=v(z_1)$ for some $z_1\in u^{-1}(1_b),$ and it follows that $(x_0,y_0)=\la u,v\ra(z_1\cdot x_1).$
\end{proof}
By the definition of $d_1,$ the $1$-direction of \eqref{eqn:1-grpd-special} is given by the kernel of $\la\un{u},\un{v}\ra\colon K(\un{f}_1)\longrightarrow  K(\un{f}_0)\times K(\un{f}_0),$ i.e. by a pullback of $\bigsqcup_{b\in B} f_1^{-1}(1_b) \longrightarrow \bigsqcup_{b\in B}(f_0^{-1}\big(1_b)\times (f_0^{-1}(1_b)\big),$ $(x\colon b\rightarrow b)\mapsto\big(u(x),v(x)\big),$ along the map $B\longrightarrow \bigsqcup_{b\in B}(f_0^{-1}\big(1_b)\times (f_0^{-1}(1_b)\big),$ $b\mapsto(1_b,1_b).$ The latter can be realized by the disjoint union $\bigsqcup_{b\in B}A(b),$ where $A(b)=\{x\in f_1^{-1}(1_b):u(x)=v(x)=1_b\},$ and corresponds to the set of arrows of the small category $\un{A}^+\in\mathbf{Cat}_B$ given by the $\un{Y}$-module $A\colon \un{Y}\rightarrow\mathbf{Ab},$ $(y\colon a\rightarrow b)\longmapsto \big(A(a)\rightarrow A(b),k\mapsto {^yk}\big),$ where ${^yk}$ is defined as $q_{\un{f}_1}(x_1, x_1\cdot k)$ with $x_1$ any arrow satisfying $f_1(x_1)=y$; since $\un{u},\un{v}$ are morphisms in $Sl(\mathbf{Cat}_B/\underline{Y})$, by \ref{lemma:q_comp} we have
\begin{equation}
	\label{eqn:is_mod}
	u({^yk}) = u (q_{\un{f}_1}(x_1, x_1\cdot k)) = q_{\un{f}_0}(u(x_1), u(x_1\cdot k)) = q_{\un{f}_0}(u(x_1), u(x_1))=1_b
\end{equation}
and similarly $v({^yk}) = 1_b$, so that ${^yk}\in A(b)$ whenever $k\in A(a).$

Thus, an aspherical groupoid \eqref{eqn:1-grpd-special} in $\mathrm{Mal}\big(Sl(\mathbf{Cat}_B/\un{Y})\big)$ with $1$-direction $A\in\mathrm{Mod}_{\un{Y}}$ corresponds to a crossed sequence
\begin{equation}
\label{eqn:1-fold-ext}
\begin{aligned}
\xymatrixrowsep{1pc}
\xymatrixcolsep{1.7pc}
\xymatrix{{\un{A}^+} \ar@{>->}[r] &{K(\un{u})} \ar@{->>}[rd] \ar[rr]^-{\un{\delta}=\un{v}_|} &\ &{\un{X}_0} \ar@{->>}[r]^-{\un{f}_0} &{\un{Y}.} \\
&\ &{K(\un{f}_0)} \ar@{>->}[ru] &\ &\
}
\end{aligned}
\end{equation}
Observe that in $K(u)=\bigsqcup_{b\in B}u^{-1}(1_b)$ we have $u^{-1}(1_b)\subseteq f_1^{-1}(1_b)\in\mathbf{Ab}$ for every $b\in B,$ and that in the above diagram both $\xymatrixcolsep{1.7pc}\xymatrix{{\un{A}^+} \ar@{>->}[r] &{K(\un{u})} \ar@{->>}[r] &{K(\un{f}_0)}}$ and $\xymatrixcolsep{1.7pc}\xymatrix{{K(\un{f}_0)} \ar@{>->}[r] &{\un{X}_0} \ar@{->>}[r]^-{\un{f}_0} &{\un{Y}}}$ are extensions of small categories in the sense of \eqref{eqn:ext_of_cat} (using the fact that $\un{v}$ is $S$-special by Remark \ref{rmks:s-special}$(3),$ since so are $\un{f}_0,$ $\un{f}_1$ and $\un{f}_1=\un{f}_0\un{v}$).

This means that \eqref{eqn:1-fold-ext} is a $1$-fold extension in the sense of \cite{golasinski}.

Conversely, consider a crossed module $\un{\delta}$ in $\mathbf{Cat}_B$ as in the diagram
\begin{equation*}
\xymatrixrowsep{1pc}
\xymatrixcolsep{1.7pc}
\xymatrix{{\un{A}^+} \ar@{>->}[r] &{\un{Z}^+} \ar@{->>}[rd] \ar[rr]^-{\un{\delta}} &\ &{\un{X}_0} \ar@{->>}[r]^-{\un{f}_0} &{\un{Y},} \\
&\ &{K(\un{f}_0)} \ar@{>->}[ru] &\ &\
}
\end{equation*}
where $\un{Z}^+$ comes from a $\un{Y}$-module $Z\colon \un{Y}\longrightarrow\mathbf{Ab}$ and both sequences
\begin{equation*}
\begin{aligned}
\xymatrixcolsep{1.7pc}\xymatrix{{\un{A}^+} \ar@{>->}[r] &{\un{Z}^+} \ar@{->>}[r] &{K(\un{f}_0),}} \ \  \xymatrixcolsep{1.7pc}\xymatrix{{K(\un{f}_0)} \ar@{>->}[r] &{\un{X}_0} \ar@{->>}[r]^-{\un{f}_0} &{\un{Y}}}
\end{aligned}
\end{equation*}
are extensions of small categories \eqref{eqn:ext_of_cat}.

Then, in $K(f_0)=\bigsqcup_{b\in B}f_0^{-1}(1_b),$ every $f_0^{-1}(1_b)$ is an abelian group, and it follows by Proposition \ref{prop:char_mal'tsev_objects_slice} that $\un{f}_0\in\mathrm{Mal}\big(Sl(\mathbf{Cat}_B/\un{Y})\big).$

Moreover, by Proposition \ref{prop:scat_xmod}, $\un{\delta}$ corresponds to a Schreier internal groupoid
\begin{equation*}
\xymatrix{
{\un{Z}^+\rtimes_{\un{\omega}}\un{X}_0} \ar[rd]_-{\un{f}_1=\un{f}_0\un{\pi}=\un{f}_0\un{\gamma}} \ar@<.9ex>[rr]^-{\un{\pi}} \ar@<-.9ex>[rr]_-{\un{\gamma}} &\ &{\un{X}_0} \ar[ll]|-{\un{\sigma}} \ar[ld]^-{\un{f}_0} \\
&{\un{Y}}
}
\end{equation*}
in $\mathbf{Cat}_B$ (as in \eqref{eqn:Sigma}), where $\un{\omega}\colon\un{X}_0\longrightarrow E(\un{Z}^+)$ is the given action in the crossed module $\un{\delta}.$

Observe that, since $Z$ is a functor on $\mathbf{Ab}$ and by assumption $K(f_0)$ coincides with the image of $\delta,$ for every $k\in f_0^{-1}(1_b)$ and every $z\in \un{Z}^+(b,b)$ we have $\omega(k)(z)=\omega\big(\delta(x)\big)(z)=x\cdot z\cdot x^{-1}=z$ for some other $x\in \un{Z}^+(b,b)$: then, if $(z\colon b\rightarrow b,x_0\colon a\rightarrow b),$ $(w\colon b\rightarrow b,r_0\colon a\rightarrow b)\in P(\un{Z}^+,\un{X}_0)$ (in the notation of Definition \ref{def:action_Cat_B}) are such that $f_0(x_0)=f_0(r_0),$ using the fact that $\un{f}_0$ is $S$-special we have $r_0=k_0\cdot x_0$ for a unique $k_0\in f_0^{-1}(1_b),$ so that $(w,r_0)=(w\cdot z^{-1},k_0)\cdot(z,x_0).$ Thus, $\un{f}_1$ is also $S$-special. Finally, the group $f_1^{-1}(1_b)=\un{Z}^+(b,b)\times f_0^{-1}(1_b)$ is abelian, as a product of abelian groups, so that $\un{f}_1\in\mathrm{Mal}\big(Sl(\mathbf{Cat}_B/\un{Y})\big).$

By \cite{golasinski}, Theorem 1.4, we conclude that the group $\mathrm{Opext}^2(\un{Y},A)=\pi_0\big(d_{1,\un{Y}}^{-1}(A)\big)$ of connected components of the fibre $d_{1,\un{Y}}^{-1}(A)$ of the $1$-dimensional direction functor
\begin{equation*}
d_1=d_{1,\un{Y}}\colon (1\text{-})\mathrm{Asp}\big(Sl(\mathbf{Cat}_B/\un{Y})\big)\longrightarrow\mathrm{Mod}_{\un{Y}}
\end{equation*}
is a realization of the third cohomology group $H^3(\un{Y},A)$ in the sense of \cite{cocat}.

The same theorem applies for the general case $n>1.$ Indeed, given an $n$-groupoid
\begin{equation}
\label{eqn:n-grpd-special}
\begin{aligned}
\xymatrixrowsep{2.5pc}
\xymatrixcolsep{2.5pc}
\xymatrix{
{\un{X}_n} \ar@/_1pc/[rrd]_-{\un{f}_n} \ar@<.9ex>[r]^-{\un{u}_n} \ar@<-.9ex>[r]_-{\un{v}_n}  &{\dots} \ar[l]|{\un{w}_n} \ar@<.9ex>[r] \ar@<-.9ex>[r] &{\un{X}_2} \ar[d]_-{\un{f}_2} \ar[l] \ar@<.9ex>[r]^-{\un{u}_2} \ar@<-.9ex>[r]_-{\un{v}_2} &{\un{X}_1} \ar[ld]^(.40){\un{f}_1} \ar[l]|-{\un{w}_2}  \ar@<.9ex>[r]^-{\un{u}_1} \ar@<-.9ex>[r]_-{\un{v}_1} &{\un{X}_0} \ar@/^1pc/[lld]^-{\un{f}_0} \ar[l]|-{\un{w}_1} \\
&\ &{\un{Y}}
}
\end{aligned}
\end{equation}
in $\mathrm{Mal}\big(Sl(\mathbf{Cat}_B/\un{Y})\big),$ for every $i=0,\dots,n$ the morphism $\un{f}_i$ is $S$-special and $f_i^{-1}(1_b)$ is an abelian group (for every $b\in B$): then, for $i=1,\dots,n,$ each $(\un{u}_i,\un{w}_i)$ is a Schreier point, and consequently $\xymatrix{ {\un{X}_i} \ar@<.9ex>[r]^-{\un{u}_i} \ar@<-.9ex>[r]_-{\un{v}_i} &{\un{X}_{i-1}} \ar[l]|-{\un{w}_i}}$ is a Schreier groupoid in $\mathbf{Cat}_B,$ corresponding by Proposition \ref{prop:scat_xmod} to the crossed module $K(\un{u}_i)\xlongrightarrow{{\un{v}_i}_|}\un{X}_{i-1}.$ Observe that, for $i=2,\dots,n,$ the image of ${\un{v}_i}_|$ lies in $K(\un{u}_{i-1})$ (using the equation $\un{u}_{i-1}\un{v}_i=\un{u}_{i-1}\un{u}_i$ given by the assumption that \eqref{eqn:n-grpd-special} is an $n$-groupoid). Similarly, the equation $\un{v}_{i-1}\un{v}_i=\un{v}_{i-1}\un{u}_i$ guarantees that ${\un{v}_{i-1}}_|{\un{v}_i}_|=0.$ Moreover, if $f_0$ is surjective, so is $f_i$ for $i=1,\dots,n,$ and each $K(\un{u}_i)$ is a $\un{Y}$-module (cf.~\eqref{eqn:is_mod}).
Then, as proven above for the case $n=1$, an aspherical $n$-groupoid $\mathbb{X}$ in $\mathrm{Mal}\big(Sl(\mathbf{Cat}_B/\un{Y})\big),$ as in \eqref{eqn:n-grpd-special}, with $n$-direction $d_{n,\un{Y}}(\mathbb{X})=A\in\mathrm{Mod}_{\un{Y}},$ corresponds to an $n$-crossed extension
\begin{equation*}
 \xymatrix{{\un{A}^+} \ar@{>->}[r] &{K(\un{u}_n)} \ar[r]^-{{\un{v}_n}_|} &{K(\un{u}_{n-1})} \ar[r] &{\dots} \ar[r] &{K(\un{u}_1)}  \ar[r]^-{{\un{v}_1}_|} &{\un{X}_0} \ar@{->>}[r]^-{\un{f}_0} &{\un{Y}} }
 \end{equation*}
 (i.e., to a $n$-fold extension of $\un{Y}$ by $A,$ in the sense of \cite{golasinski}), and by \cite{golasinski}, Theorem 1.4, we have that $\mathrm{Opext}^{n+1}(\un{Y},A)=\pi_0\big(d_{n,\un{Y}}^{-1}(A)\big)\cong H^{n+2}(\un{Y},A).$

\subsection{A remark on crossed extensions in $\mathbf{Cat}_B$}
Let us call \emph{crossed extension} in $\mathbf{Cat}_B$ a diagram
\begin{equation} \label{crossed extension}
\xymatrixrowsep{1pc}
\xymatrixcolsep{1.7pc}
\xymatrix{{\un{A}} \ar@{>->}[r]^{\un{j}} &{\un{Z}} \ar@{->>}[rd]_{\un{\beta}} \ar[rr]^-{\un{\delta}} &\ &{\un{X}} \ar@{->>}[r]^-{\un{f}} &{\un{Y},} \\
&\ &{K(\un{f})} \ar@{>->}[ru]_{\un{\gamma}} &\ &\
}
\end{equation}
in which $\un{\delta}$, together with an action of $\un{X}$ on $\un{Z}$, is a crossed module, $(\un{\beta}, \un{\gamma})$ is the (regular epimorphism, monomorphism) factorization of $\un{\delta}$, $(\un{A},\un{j})$ is a kernel of $\un{\delta}$ in the sense of \eqref{eqn:K_f}, and the sequence
\[ \xymatrix{ {K(\un{f})} \ar@{>->}[r]^-{\un{\gamma}} & {\un{X}} \ar@{->>}[r]^-{\un{f}} &{\un{Y}} } \]
is a special Schreier extension. We want to show that every crossed extension is equivalent to one in which ${\un{Z}}$ is an abelian groupoid. By equivalent we mean that there exists a zig-zag of morphisms of crossed extensions between the two. A morphism of crossed extensions is a commutative diagram
\[ \xymatrix{ \un{A} \ar@{=}[d] \ar[r]^{\un{j'}} &{\un{Z'}} \ar[d]_{\un{\xi}} \ar[r]^-{\un{\delta'}} & \un{X'} \ar[d]^{\un{\tau}} \ar[r]^{\un{f'}} & \un{Y'} \ar@{=}[d] \\
{\un{A}} \ar[r]_{\un{j}} &{\un{Z}} \ar[r]_-{\un{\delta}} &{\un{X}} \ar[r]_-{\un{f}} &{\un{Y}} } \]
such that the pair $\un{\xi}, \un{\tau}$ is a morphism of crossed modules, as in Definition \ref{def:mor_of_xsmod}.\\

Let \eqref{crossed extension} be a crossed extension. The morphism $\un{f}$ being a regular epimorphism in $\mathbf{Cat}_B$, it admits a section in the category of graphs over $B$. Fixing such a section $\un{s}$, the fact that $\un{f}$ is special Schreier gives us a map
\[ \varphi \colon Y \times_B Y \to K(\un{f}) \]
such that, for all $y_1\colon b_1\to b_2$ and $y_2\colon b_2\to b_3$, $\varphi(y_2,y_1)\in K(f)(b_3,b_3)$ satisfies
\begin{equation} \label{eq:def_varphi}
	\gamma(\varphi(y_2,y_1)) \cdot s(y_2\cdot y_1) = s(y_2)\cdot s(y_1).
\end{equation}
Moreover
\[\varphi(y,1_b) = 1_{b'} = \varphi(1_{b'},y)\]
for all $y\colon b\to b'$ in $Y$. For every composable triple $y_1, y_2, y_3 \in Y$ we can compute $s(y_3) \cdot s(y_2) \cdot s(y_1)$ in two possible ways:
\begin{align*}
	(s(y_3)\cdot s(y_2)) \cdot s(y_1) & = \gamma(\varphi(y_3,y_2))\cdot s(y_3\cdot y_2) \cdot s(y_1)\\
	& = \gamma(\varphi(y_3,y_2))\cdot  \gamma(\varphi(y_3 y_2,y_1)) \cdot s(y_3\cdot y_2 \cdot y_1)
\end{align*}
and
\begin{align*}
	s(y_3)\cdot (s(y_2)\cdot s(y_1)) & = s(y_3)\cdot \gamma(\varphi(y_2,y_1))\cdot s(y_2\cdot y_1)\\
	& = \gamma({}^{s(y_3)}\varphi(y_2,y_1))\cdot s(y_3) \cdot s(y_2y_1) \\
	& = \gamma({}^{s(y_3)}\varphi(y_2,y_1)) \cdot \gamma(\varphi(y_3,y_2y_1)) \cdot s(y_3\cdot y_2 \cdot y_1),
\end{align*}
where by ${}^{s(y_3)}\varphi(y_2,y_1)$ we denote the action of $\un{Y}$ on $K(\un{f})$ determined by the special Schreier extension $\un{f}$.

Using associativity, we get
\begin{equation} \label{eq:id_varphi}
	\varphi (y_3,y_2)\cdot \varphi (y_3y_2,y_1) = {}^{s(y_3)}\varphi(y_2,y_1) \cdot\varphi (y_3,y_2y_1).
\end{equation}

The surjectivity of $\un{\beta}$ gives a map $\psi \colon Y \times_B Y \to Z$ such that $ \beta \psi = \phi$,
\[\psi(y,1_b) = 1_{b'} = \psi(1_{b'},y)\]
for all $y\colon b\to b'$ in $Y$ and, for every composable triple $y_1, y_2, y_3 \in Y,$ there exists a unique $\zeta(y_3, y_2, y_1) \in A$ such that
\begin{equation} \label{eq:def_zeta}
	j(\zeta (y_3, y_2, y_1)) \cdot \psi (y_3,y_2) \cdot \psi (y_3y_2,y_1) = {}^{s(y_3)}\psi(y_2,y_1) \cdot \psi (y_3,y_2y_1).
\end{equation}
This defines a map
\[ \zeta \colon Y \times_B Y \times_B Y \to A \]
which satisfies the following identity for every chain $y_1, y_2, y_3, y_4$ of composable arrows:
\begin{equation} \label{eq:cocycle}
	\zeta(y_4y_3,y_2,y_1) + \zeta(y_4,y_3,y_2y_1) = {}^{y_4}\zeta(y_3,y_2,y_1) + \zeta (y_4,y_3y_2,y_1) + \zeta(y_4,y_3,y_2).
\end{equation}
To show this, we compute in two ways the expression
\[ \diamond = {}^{s(y_4)s(y_3)}\psi(y_2,y_1) \cdot {}^{s(y_4)}\psi(y_3,y_2y_1) \cdot \psi(y_4,y_3y_2y_1). \]
On one hand, using \eqref{eq:def_zeta} repeatedly, we have
\begin{align*}
	\diamond & = j(\zeta (y_4, y_3, y_2y_1)) \cdot {}^{s(y_4)s(y_3)}\psi(y_2,y_1) \cdot \psi (y_4,y_3) \cdot \psi (y_4y_3,y_2y_1) \\
	& =  j(\zeta (y_4, y_3, y_2y_1)) \cdot  {}^{\partial(\psi(y_4,y_3))s(y_4y_3)}\psi(y_2,y_1) \cdot \psi (y_4,y_3) \cdot \psi (y_4y_3,y_2y_1) \\
	& = j(\zeta (y_4, y_3, y_2y_1)) \cdot \psi(y_4,y_3)\cdot {}^{s(y_4y_3)} \psi(y_2,y_1) \cdot \psi (y_4y_3,y_2y_1) \\
	& = j(\zeta(y_4y_3,y_2,y_1) + \zeta (y_4, y_3, y_2y_1)) \cdot \psi(y_4,y_3) \cdot \psi(y_4y_3,y_2) \cdot \psi(y_4y_3y_2,y_1),
\end{align*}
while, on the other hand, we have
\begin{align*}
	\diamond & = {}^{s(y_4)} ({}^{s(y_3)} \psi(y_2,y_1) \cdot \psi(y_3,y_2y_1)) \cdot \psi(y_4,y_3y_2y_1) \\
	& = {}^{s(y_4)} (j(\zeta(y_3,y_2,y_1)) \cdot \psi(y_3,y_2) \cdot \psi(y_3y_2,y_1)) \cdot \psi(y_4,y_3y_2y_1) \\
	& = j({}^{y_4}\zeta(y_3,y_2,y_1)) \cdot {}^{s(y_4)} \psi(y_3,y_2) \cdot {}^{s(y_4)} \psi(y_3y_2,y_1) \cdot \psi(y_4,y_3y_2y_1) \\
	& = j({}^{y_4}\zeta(y_3,y_2,y_1) + \zeta(y_4,y_3y_2,y_1)) \cdot {}^{s(y_)4} \psi(y_3,y_2) \cdot \psi(y_4,y_3y_2) \cdot \psi(y_4y_3y_2,y_1) \\
	& = j({}^{y_4}\zeta(y_3,y_2,y_1) + \zeta(y_4,y_3y_2,y_1) + \zeta(y_4,y_3,y_2)) \cdot \psi(y_4,y_3) \cdot \psi(y_4y_3,y_2) \cdot \psi(y_4y_3y_2,y_1).
\end{align*}

Let then $\un{F}$ be the totally disconnected groupoid over $B$ such that, for $b_3 \in B$, $F(b_3, b_3)$ is the free group generated by pairs $[y_2,y_1]$ of arrows $y_1\colon b_1\to b_2$ and $y_2\colon b_2\to b_3$ of $Y$, with the convention that $[y_2,y_1]$ is the empty word if either $y_1$ or $y_2$ is an identity. For $y_3 \colon b_3 \to b_4$, let us define
\begin{equation} \label{eq:def_act}
	^{y_3}[y_2,y_1] = [y_3,y_2] \cdot [y_3y_2,y_1] \cdot [y_3,y_2y_1]^{-1}
\end{equation}
and let us extend it freely, in order to get a group homomorphism ${}^{y} (\_ )$ for all $y \in Y$. We observe that for all $\xymatrix{ b_1 \ar[r]^{y_1} & b_2 \ar[r]^{y_2} & b_3 \ar[r]^{y_3} & b_4 \ar[r]^{y_4} & b_5 }$, we have
\begin{align*}
	^{y_4} ({ ^{y_3} [y_2,y_1]}) = & {}^{y_4} ([y_3,y_2] \cdot [y_3y_2,y_1] \cdot [y_3,y_2y_1]^{-1}) \\
	= & [y_4,y_3] \cdot [y_4y_3,y_2] \cdot [y_4,y_3y_2]^{-1}\\
	& \cdot [y_4,y_3y_2] \cdot [y_4y_3y_2,y_1] \cdot [y_4,y_3y_2y_1]^{-1} \\
	& \cdot ([y_4,y_3] \cdot [y_4y_3,y_2y_1] \cdot [y_4,y_3y_2y_1]^{-1})^{-1}\\
	= &  [y_4,y_3] \cdot [y_4y_3,y_2] \cdot [y_4,y_3y_2]^{-1}\\
	& \cdot [y_4,y_3y_2] \cdot [y_4y_3y_2,y_1] \cdot [y_4,y_3y_2y_1]^{-1} \\
	& \cdot [y_4,y_3y_2y_1]  \cdot  [y_4y_3,y_2y_1]^{-1} \cdot  [y_4,y_3]^{-1} \\
	= &  [y_4,y_3] \cdot [y_4y_3,y_2] \cdot [y_4y_3y_2,y_1] \cdot [y_4y_3,y_2y_1]^{-1} \cdot  [y_4,y_3]^{-1},
\end{align*}
which implies that
\begin{equation} \label{eq:act_CF}
	[y',y] \cdot {}^{y'\cdot y} w = {}^{y'}({}^{y} w) \cdot [y',y]
\end{equation}
for all $w\in F(b_1, b_1)$ and all composable $y, y' \in Y.$ We define $\widehat{\varphi} \colon F \to K(f)$ as the unique morphism of groups such that $\widehat{\varphi}([y_2,y_1]) = \varphi(y_2,y_1)$. We similarly define $\widehat{\psi} \colon F \to X $ such that $\widehat{\psi}([y_2,y_1])= \psi (y_2,y_1),$ and $\widehat{\zeta}\colon Y\times_B F\to A$ in such a way that for all $y_3\colon b_3\to b_4$ in $Y,$ $v\mapsto\zeta(y_3,v)$ is the unique morphism of groups $F(b_3, b_3) \to F(b_4, b_4)$ such that $\widehat{\zeta} (y_3,[y_2,y_1]) = \zeta(y_3,y_2,y_1)$ for all $\xymatrix{ b_1 \ar[r]^{y_1} & b_2 \ar[r]^{y_2} & b_3. }$ Note that for all $y,w$, we have
\[\widehat{\varphi}({}^{y}w) = {}^{s(y)}\widehat{\varphi}(w),\]
since this holds by construction whenever $w=[y_2,y_1],$ and both sides define group morphisms on the connected components of $\un{F}$ (when $y$ is fixed).

Next, we define $\un{\widehat{X}}$ so that $\widehat{X}(b_1,b_2)=F(b_2, b_2) \times Y(b_1,b_2),$ with composition defined by
\[(w_2,y_2)\cdot (w_1,y_1) = (w_2({}^{y_2} w_1)[y_2,y_1] , y_2y_1). \]
The definition of ${}^y (\_ )$ ensures that this composition is associative. Moreover we define $\un{\widehat{f}}\colon \un{\widehat{X}} \to \un{Y}$ as the projection. We then have a commutative diagram
\begin{equation} \label{first part zigzag}
\xymatrix{ \un{A} \ar@{=}[d] \ar[r]^-{\langle 0, 1 \rangle} & \un{F} \times \un{A} \ar[d]_{\un{\xi}} \ar[r]^-{\un{\widehat{\delta}}} & \un{\widehat{X}} \ar[d]^{\un{\tau}} \ar[r]^{\un{\widehat{f}}} & \un{Y} \ar@{=}[d] \\
{\un{A}} \ar[r]_{\un{j}} &{\un{Z}} \ar[r]_-{\un{\delta}} &{\un{X}} \ar[r]_-{\un{f}} &{\un{Y},} }
\end{equation}
where $\widehat{\delta}(w,a) = (w,1)$, $\tau(w,y) = \gamma(\widehat{\varphi}(w))\cdot s(y)$ and $\xi (w,a) = \widehat{\psi}(w)\cdot j(a)$. Let us check that $\un{\tau}$ is indeed a morphism: we have
\begin{align*}
	\tau((w_2,y_2)\cdot (w_1,y_1)) & = \tau(w_2({}^{y_2} w_1)[y_2,y_1] , y_2y_1) = \gamma(\widehat{\varphi} (w_2({}^{y_2} w_1)[y_2,y_1]))\cdot s(y_2y_1)\\
	& = \gamma(\widehat{\varphi}(w_2)) \gamma(\widehat{\varphi}({}^{y_2} w_1)) \gamma(\varphi(y_2,y_1))\cdot s(y_2y_1)\\
	& = \gamma(\widehat{\varphi}(w_2))  \gamma(\widehat{\varphi}({}^{y_2} w_1)) s(y_2) s(y_1) \\
	& = \gamma(\widehat{\varphi}(w_2))  \gamma({}^{s(y_2)}\widehat{\varphi}( w_1)) s(y_2) s(y_1) \\
	& = \gamma(\widehat{\varphi}(w_2)) s(y_2) \gamma(\widehat{\varphi}( w_1)) s(y_1)\\
	& = \tau(w_2,y_2)\cdot \tau(w_1,y_1).
\end{align*}

Now we define an action of $\un{\widehat{X}}$ on the totally disconnected groupoid $\un{F} \times \un{A}$ by putting
\[{}^{(w,y)} (v,a) = (w{}^y vw^{-1}, \widehat{\zeta}(y,v) + {}^{y}a).\]
In order to show that this defines an action, the only non trivial equation to check is
\[{}^{(w_2,y_2)\cdot (w_1,y_1)} (v,a) = {}^{(w_2,y_2)} \Big({}^{(w_1,y_1)} (v,a)\Big).\]
The first term is
\[{}^{(w_2{}^{y_2}w_1 [y_2,y_1],y_2y_1)} (v,a) = (w_2{}^{y_2}w_1 [y_2,y_1]  {}^{y_2y_1}v (w_2{}^{y_2}w_1 [y_2,y_1])^{-1} , \widehat{\zeta} (y_2y_1,v) + {}^{y_2y_1}a)\]
while the second term is
\begin{gather*}
	{}^{(w_2,y_2)} (w_1{}^{y_1}v w_1^{-1}, \widehat{\zeta}(y_1,v) + {}^{y_1}a) \\
	= \big(w_2 {}^{y_2} w_1 {}^{y_2}({}^{y_1}v) {}^{y_2}(w_1^{-1}) w_2^{-1} , \widehat{\zeta} (y_2, w_1{}^{y_1}v w_1^{-1} ) +{}^{y_2} (\widehat{\zeta}(y_1,v) + {}^{y_1}a)\big).
\end{gather*}
By comparing the two, and since $\un{Y}$ acts on $\un{A}$ and $\un{A}$ is abelian, we conclude that it suffices to check the equality
\[\widehat{\zeta} (y_2y_1,v) = \widehat{\zeta} (y_2,{}^{y_1}v) + {}^{y_2}\widehat{\zeta} (y_1,v).\]
Actually, it is enough to check this identity on the generators $v=[z_2,z_1]$ of $F(b_1,b_1),$ where $b_1$ is the domain of $y_1.$ We have
\begin{align*}
	\zeta(y_2y_1,z_2,z_1) & = \widehat{\zeta}(y_2,[y_1,z_2] \cdot [y_1z_2,z_1] \cdot [y_1,z_2z_1]^{-1}) + {}^{y_2} \zeta(y_1,z_2,z_1)\\
	& = \zeta(y_2,y_1,z_2) + \zeta(y_2,y_1z_2,z_1) - \zeta (y_2,y_1,z_2z_1) + {}^{y_2} \zeta(y_1,z_2,z_1),
\end{align*}
which is equivalent to the identity \eqref{eq:cocycle}.

The morphism $\un{\widehat{\delta}},$ with the above defined action, is a crossed module. Indeed the first condition for crossed modules corresponds to
\[\widehat{\delta} (w {}^{y}v v^{-1}, \widehat{\zeta}(y,v) + {}^{y}a) (w,y) = (w,y)\widehat{\delta}(v,a),\]
which amounts to
\[w^{y}v w^{-1}w = w{}^{y}v,\]
which is clearly true. The second condition is satisfied as well, since when $y=1$ we get
\[{}^{(w,1)} (v,a) = (w vw^{-1},a) = (w,1)(v,a)(w,1)^{-1}.\]
Let us now show that $(\un{\xi}, \un{\tau})$ is a morphism of crossed modules. The equivariance with respect to the action amounts to the condition
\[\xi({}^{(w,y)} (v,a)) = {}^{\tau (w,y)} \xi(v,a),\]
which is the same as
\[\widehat{\psi} (w{}^y vw^{-1}) \cdot j (\widehat{\zeta}(y,v) + {}^{y}a) = {}^{\gamma(\widehat{\varphi}(w))s(y)} \big(\widehat{\psi}(v) \cdot j(a)\big),\]
or equivalently
\[\widehat{\psi} (w) \widehat{\psi}({}^y v ) \widehat{\psi}(w^{-1}) \cdot j (\widehat{\zeta}(y,v) + {}^{y}a) = {}^{\partial(\widehat{\psi}(w))s(y)} (\widehat{\psi}(v) \cdot j(a) ) = \widehat{\psi} (w) {}^{s(y)}\widehat{\psi}(v) \widehat{\psi}(w^{-1}) \cdot j ({}^{y}a),\]
so it suffices to check that
\[\widehat{\psi}({}^y v ) \cdot j (\widehat{\zeta}(y,v)) = {}^{s(y)}\widehat{\psi}(v).\]
Checking this on a generator $v=[y_2,y_1]$ is equivalent to checking that
\[\psi(y,y_2) \cdot \psi (yy_2,y_1) \cdot \psi(y,y_2y_1)^{-1} \cdot j(\zeta(y,y_2,y_1)) = {}^{s(y)} \psi(y_2,y_1),\]
which holds by definition of $\zeta.$ So we conclude that \eqref{first part zigzag} is a morphism of crossed extensions.

Considering then the groupoid $\un{G}$ obtained from $\un{F}$ by taking the abelianization of all the endogroups of $\un{F}$, and the category $\un{\tilde{X}}$ obtained from $\un{\widehat{X}}$ by replacing $F$ with $G$, we end up with the following zig-zag of crossed extensions:
\[ \xymatrix{ \un{A} \ar[r]^-{\langle 0, 1 \rangle} & \un{G} \times \un{A} \ar[r]^-{\un{\tilde{\delta}}} & \un{\tilde{X}} \ar[r]^{\un{\tilde{f}}} & \un{Y} \\
\un{A} \ar@{=}[d] \ar@{=}[u] \ar[r]^-{\langle 0, 1 \rangle} & \un{F} \times \un{A} \ar[d]_{\un{\xi}} \ar[u]^{\un{\pi} \times 1} \ar[r]^-{\un{\widehat{\delta}}} & \un{\widehat{X}} \ar[d]^{\un{\tau}} \ar[u]_{\un{\tilde{\pi}}} \ar[r]^{\un{\widehat{f}}} & \un{Y} \ar@{=}[d] \ar@{=}[u] \\
{\un{A}} \ar[r]_{\un{j}} &{\un{Z}} \ar[r]_-{\un{\delta}} &{\un{X}} \ar[r]_-{\un{f}} &{\un{Y},}
} \]
where $\un{\pi}$ and $\un{\tilde{\pi}}$ are the quotient projections, and $\un{\tilde{\delta}}$ and $\un{\tilde{f}}$ are obtained by the universal property of the quotients.

\section*{Acknowledgement}
The first and third authors are members of the Gruppo Nazionale per le Strutture Algebriche, Geometriche e le loro Applicazioni (GNSAGA) dell'Istituto Nazionale di Alta Matematica ``Francesco Severi''.

The second author is a Postdoctoral researcher of the Fonds de la Recherche Scientifique-FNRS.

This work was supported by the Shota Rustaveli National Science Foundation of Georgia (SRNSFG), through grant FR-24-9660, ``Categorical methods for the study of cohomology theory of monoid-like structures: an approach through Schreier extensions''.


\end{document}